\documentclass[12pt,leqno]{article}

\usepackage{amsmath,amsthm,amssymb,mathrsfs,stmaryrd,color,hyperref}
\usepackage[all]{xy}
\usepackage{url}
\usepackage{slashed}
\usepackage{geometry}
\usepackage{tikz}
\usepackage{tikz-cd}
\usepackage[all]{xy}
\usepackage{relsize}
\usepackage[bbgreekl]{mathbbol}

\usepackage{color}

\DeclareSymbolFontAlphabet{\mathbb}{AMSb}
\DeclareSymbolFontAlphabet{\mathbbl}{bbold}

\renewenvironment{proof}{{\noindent\it\textbf{Proof}}\quad\it}{\hfill$\square$}

\newtheorem{thm}[equation]{Theorem}
\newtheorem{cor}[equation]{Corollary}
\newtheorem{lem}[equation]{Lemma}
\newtheorem{prop}[equation]{Proposition}

\newtheorem{ex}[equation]{Example}
\newtheorem{convention}[equation]{Convention}
\newtheorem{construction}[equation]{Construction}

\theoremstyle{definition}
\newtheorem{defn}[equation]{Definition}
\newtheorem{rem}[equation]{{Remark}}

\numberwithin{equation}{section}

\newcommand{\Prism}{\mathbbl{\Delta}}
\newcommand{\Sets}{\mathrm{Sets}}

\newcommand{\Ab}{\mathrm{Ab}}

\newcommand{\Z}{\mathbb{Z}}
\newcommand{\C}{\mathbb{C}}

\newcommand{\Q}{\mathbb{Q}}

\newcommand{\N}{\mathbb{N}}
\newcommand{\Coeq}{\mathrm{Coeq}}
\newcommand{\F}{\mathbb{F}}

\newcommand{\Hom}{\mathrm{Hom}}

\newcommand{\LCoMod}{\mathrm{LCoMod}}
\newcommand{\RCoMod}{\mathrm{RCoMod}}

\newcommand{\TC}{\mathrm{TC}}
\newcommand{\THH}{\mathrm{THH}}
\newcommand{\TP}{\mathrm{TP}}
\newcommand{\Sphere}{\mathbb{S}}

\newcommand{\PfAb}{\mathrm{PfAb}}

\newcommand{\QCoh}{\mathrm{QCoh}}
\newcommand{\Shv}{\mathrm{Shv}}
\newcommand{\Mod}{\mathrm{Mod}}
\newcommand{\op}{\mathrm{op}}

\newcommand{\CAlg}{\mathrm{CAlg}}
\newcommand{\id}{\mathrm{id}}

\newcommand{\Ext}{\mathrm{Ext}}
\newcommand{\Gal}{\mathrm{Gal}}

\newcommand{\Aut}{\mathrm{Aut}}

\newcommand{\RHom}{\mathrm{RHom}}

\newcommand{\Ker}{\mathrm{Ker}}
\newcommand{\Spec}{\mathrm{Spec}}

\title{On Dual Algebras of Hopf Algebroids}
\author{Jingbang Guo \thanks{The author would thank the support from Guozhen Wang and the funding NSFC-12226002.}}

\begin{document}

\sffamily
\maketitle
\thispagestyle{empty}

\abstract{We study the dual algebras of (discrete) Hopf algebroids. In particular, we understand comodules over a Hopf algebroid as (discrete) modules over its dual algebra.}
\tableofcontents

\large

\begin{sloppy}
\clubpenalty=-100
\widowpenalty=-100

\section{Introduction}\label{section:introduction}

This elementary article attempts to understand comodules over a Hopf algebroid as modules over the corresponding dual algebra.


\begin{rem}\label{rem:Hopf-algebroids}
    A Hopf algebroid is a pair of commutative rings $(A,\Gamma)$, together with ring homomorphisms $\eta_L,\eta_R:A\rightarrow \Gamma$, $\epsilon:\Gamma\rightarrow A$, $\Delta:\Gamma\rightarrow \Gamma\otimes_{\eta_R,A,\eta_L}\Gamma$, and $c:\Gamma\rightarrow\Gamma$, so that the pair of functors $\Hom_{\CAlg}(A,-),\Hom_{\CAlg}(\Gamma,-):\CAlg\rightarrow \Sets$ together form a functor $\CAlg\rightarrow \textrm{Groupoids}$, such that $\Hom_\CAlg(A,-)$ is the functor of objects, and $\Hom_\CAlg(\Gamma,-)$ is the functor of morphisms, and $\eta_L,\eta_R$ correspond to the source and target, $\epsilon$ corresponds to the identity, $\Delta$ corresponds to the composition of morphisms, and finally, $c$ corresponds to the inversion of morphisms. For more details on Hopf algebroids, see \cite[Appendix A1]{Ravenel04}.
\end{rem}

\begin{rem}\label{rem:Hopf-algebroids-and-stacks}
    To each Hopf algebroid there is the associated stack: heuristically, a pair of rings $(A,\Gamma)$ together with the structure morphisms give rise to a geometric object $\Coeq(\textbf{Spec}(\Gamma)\rightrightarrows\textbf{Spec}(A))$, where the two arrows are induced by $\eta_L$ and $\eta_R$. Conversely, for a suitable stack $\mathfrak{X}$ with a suitable affine cover $\Spec(A)\rightarrow \mathfrak{X}$ (see \cite[Definition 6]{Naumann07}), there would be a pullback square of the form
    $$
    \xymatrix@R=50pt@C=50pt{
    \textbf{Spec}(\Gamma) \ar[r]^{\Spec(\eta_R)} \ar[d]^{\Spec(\eta_L)} & \textbf{Spec}(A) \ar[d]^{e} \\
    \textbf{Spec}(A)      \ar[r]^{e}        & \mathfrak{X},
    }
    $$
    making $(A,\Gamma)$ a Hopf algebroid. In fact, the $2$-category of flat Hopf algebroids and the $2$-category of rigidified algebraic stacks are equivalent to each other, see \cite[Theorem 8]{Naumann07}.
\end{rem}

\begin{rem}\label{rem:comodules-and-quasi-coherent-sheaves}
    Let $(A,\Gamma)$ be a flat Hopf algebroid (in the sense that $\eta_L:A\rightarrow \Gamma$ is flat), and $\Spec(A)\rightarrow \mathfrak{X}$ the associated rigidified algebraic stack. A quasi-coherent sheaf on $\mathfrak{X}$ can then be interpreted as an $A$-module with descent data, that is to say, a quasi-coherent sheaf on $\mathfrak{X}$ can be interpreted as a (right) $\Gamma$-comodule: a $\Gamma$-comodule is an $A$-module $M$ equipped with a comodule structure map $\psi_M:M\rightarrow M\otimes_{A,\eta_L}\Gamma$, which is $A$-linear with respect to the $\eta_R$-structure of $\Gamma$, and satisfies co-associativity and co-unitarity (See \cite[Definition A1.1.2]{Ravenel04}). In fact, there is an equivalence of categories 
    $$
    \QCoh(\mathfrak{X})\simeq \RCoMod_\Gamma,
    $$
    through which quasi-coherent sheaf cohomology corresponds to comodule cohomology. See \cite[$\S$3.4]{Naumann07}.
\end{rem}

\begin{rem}\label{rem:topological-Hochschild-homology}
    This article was originally required by an investigation into computations of topological cyclic homology, for which Hopf algebroids naturally appears. For example, in \cite{LW22}, by considering the descent along $\THH(\mathcal{O}_K/\Sphere_{W(k)})\rightarrow \THH(\mathcal{O}_K/\Sphere_{W(k)}[z])$, where $\mathcal{O}_K$ is the ring of integers of a local field $K$ with residue field $k$, the computation of $\TC(\mathcal{O}_K/\Sphere_{W(k)})$ is reduced to computations concerning the comodule cohomology of Hopf algebroids $(\THH_*(\mathcal{O}_K/\Sphere_{W(k)}[z]),\THH_*(\mathcal{O}_K/\Sphere_{W(k)}[z_0,z_1])$, and $(\TP_0(\mathcal{O}_K/\Sphere_{W(k)}[z]),\TP_0(\mathcal{O}_K/\Sphere_{W(k)}[z_0,z_1])$.   
\end{rem}

\begin{rem}\label{rem:prismatic-cohomology}
    The Hopf algebroids mentioned in Remark \ref{rem:topological-Hochschild-homology} can be described algebraically with the notion of prismatic envelopes (see \cite[Proposition 3.13]{BS22}): the Hopf algebroid $(\TP_0(\mathcal{O}_K/\Sphere_{W(k)}[z]),\TP_0(\mathcal{O}_K)/\Sphere_{W(k)}[z_0,z_1])$ can be identified as 
    $$
    (W(k)[[z]],W(k)[[z_0,z_1]]\{\frac{z_1^p-z_0^p}{\varphi(E_K(z_0))}\}^\wedge_{\mathcal{N}}).
    $$
    In view of Remarks \ref{rem:Hopf-algebroids-and-stacks} and \ref{rem:comodules-and-quasi-coherent-sheaves}, conceptually there should be a stack (ignoring the possible formal geometry) encoding all the information, and different covers of this stack would provide different coordiantes for computation. For systematic treatments on the geometrization of $p$-adic cohomology theories, see for example \cite{BL22}.   
\end{rem}

\begin{rem}\label{rem:q-de-Rham}
    Thus, following \cite{LW22}, $\TC(\Z_p[\zeta_p]/\Sphere_{\Z_p})$ can be computed through the descent along $\THH(\Z_p[\zeta_p]/\Sphere_{\Z_p})\rightarrow \THH(\Z_p[\zeta_p]/\Sphere_{\Z_p}[q])$ (where $\Sphere_{\Z_p}[q]\rightarrow \Z_p[\zeta_p]$ sends $q$ to a primitive $p$-th root of unity; note that when $p=2$, it would be more convenient to consider a $4$-th root of unity); and this article was intended to provide the algebraic nonsense for this computation. In this introduction, we will illustrate the main idea with the Hopf algebroid
    $$
    (\Z_p[[q-1]],\Z_p[[q_0-1,q_1-1]]\{\frac{q_1-q_0}{[p]_{q_0}}\}_\delta),
    $$
    where $p$ is an odd prime and $[p]_{q_0}=\frac{q_0^p-1}{q_0-1}$, and the structure morphisms are 
    \begin{itemize}
        \item $\eta_L$ sends $q$ to $q_0$, $\eta_R$ sends $q$ to $q_1$;

        \item $\Delta:\Z_p[[q_0-1,q_1-1]]\{\frac{q_1-q_0}{[p]_{q_0}}\}_\delta\rightarrow \Z_p[[q_0-1,q_1-1,q_2-1]]\{\frac{q_1-q_0}{[p]_{q_0}},\frac{q_2-q_0}{[p]_{q_0}}\}_\delta$ sends $q_0$ to $q_0$, $q_1$ to $q_2$. 

        \item $\epsilon: \Z_p[[q_0-1,q_1-1]]\{\frac{q_1-q_0}{[p]_{q_0}}\}_\delta\rightarrow \Z_p[[q-1]]$ sends $q_0,q_1$ to $q$. 

        \item $c:\Z_p[[q_0-1,q_1-1]]\{\frac{q_1-q_0}{[p]_{q_0}}\}_\delta\rightarrow \Z_p[[q_0-1,q_1-1]]\{\frac{q_1-q_0}{[p]_{q_0}}\}_\delta$ interchanges $q_0$ and $q_1$. 
    \end{itemize}
\end{rem}

Let $(A,\Gamma)$ be a Hopf algebroid, and regard $\Gamma$ as an $A$-module through $\eta_L:A\rightarrow \Gamma$. The dual algebra of $(A,\Gamma)$ is $\Gamma^\vee:=\Hom_{A,\eta_L}(\Gamma,A)$.

\begin{rem}\label{rem:the-algebra-structure-of-the-dual-algebra}
    The algebra structure of $\Gamma^\vee$ is determined by the co-multiplication $\Delta:\Gamma\rightarrow \Gamma\otimes_{\eta_R,A,\eta_L}\Gamma$: given $D_1,D_2\in \Gamma^\vee$, that is, given $A,\eta_L$-linear map $D_1,D_2:\Gamma\rightarrow A$, $D_2\circ D_1$ is then the following composite
    $$
    D_2\circ D_1:\Gamma \xrightarrow{\Delta}\Gamma\otimes_{\eta_R,A,\eta_L}\Gamma\xrightarrow{\id_\Gamma\otimes D_1} \Gamma\xrightarrow{D_2}A. 
    $$
    Also, the multiplicative unit $1\in \Gamma^\vee$ is determined as the co-unit $\epsilon:\Gamma\rightarrow A$, which induces a ring homomorphism $\epsilon^\vee:A\rightarrow \Gamma^\vee$. Note that in general the algebra $\Gamma^\vee$ is not commutative.
\end{rem}

\begin{ex}\label{ex:the-prismatic-dual}
    Consider the Hopf algebroid $(A,\Gamma)=(\Z_p[[q-1]],\Z_p[[q_0-1,q_1-1]]\{\frac{q_1-q_0}{[p]_{q_0}}\}_\delta)$, together with the following ring homomorphism 
    $$
    m_q:\Z_p[[q_0-1,q_1-1]]\{\frac{q_1-q_0}{[p]_{q_0}}\}_\delta \rightarrow \Z_p[[q-1]],\quad\quad q_0\mapsto q,q_1\mapsto q^{1+p},
    $$
    which is well-defined by the universal property of prismatic envelopes (see \cite[Proposition 3.13]{BS22}). In particular, $m_q\in \Gamma^\vee$. Moreover, it is not difficult to check that, for any $f\in \Z_p[[q_0-1,q_1-1]]\{\frac{q_1-q_0}{[p]_{q_0}}\}$, $m_q(f)-f\in \Z_p[[q-1]]$ is divisible by $q-1$, and there is an $A,\eta_L$-linear map $\nabla_q:=\frac{1}{q(q-1)}\circ(m_q-1):\Gamma\rightarrow A$, or an element $\nabla_q\in \Gamma^\vee$. It turns out that, up to topology, $\Gamma^\vee$ is a $\Z_p[[q-1]]$-algebra (through $\epsilon ^\vee:A\rightarrow \Gamma^\vee$) generated by $\nabla_q$, with the relation that $\nabla_q\circ q= q^{1+p}\circ \nabla_q +[p]_q$, or 
    $$
    \nabla_q\circ f(q)=f(q^{1+p})\circ\nabla_q+\frac{f(q^{1+p})-f(q)}{q(q-1)},\quad\quad f(q)\in \Z_p[[q-1]]. 
    $$
\end{ex}

\begin{rem}\label{rem:the-prismatic-dual}
    Through $\eta_L:\Z_p[[q-1]]\xrightarrow{q\mapsto q_0}\Z_p[[q_0-1,q_1-1]]\{\frac{q_1-q_0}{[p]_{q_0}}\}_\delta$, $\Z_p[[q_0-1,q_1-1]]\{\frac{q_1-q_0}{[p]_{q_0}}\}_\delta$ is a topologically free $\Z_p[[q-1]]$-module with a basis $\{\prod_{i=0}^\infty (\delta^i(\frac{q_1-q_0}{[p]_{q_0}}))^{n_i}\}$, where $0\leq n_i<p$ and all but finitely many $n_i$'s are zero. Let $\partial_q\in \Gamma^\vee$ be the dual basis of $\frac{q_1-q_0}{[p]_{q_0}}\in \Gamma$, then $\Gamma^\vee$ can also be generated by $\Z_p[[q-1]]$ and $\partial_q$; however, it is far from immediate to determine the multiplicative relations between a general $f(q)$ and $\partial_q$. 
\end{rem}

\begin{rem}\label{rem:dual-algebra-as-an-endomorphism-algebra}
    According to the canonical isomorphism (see \cite[Definition A1.2.1]{Ravenel04}) 
    $$
    \Gamma^\vee=\Hom_{A,\eta_L}(\Gamma,A)\simeq \Hom_{\LCoMod_\Gamma}(\Gamma,\Gamma),
    $$
    the algebra structure of $\Gamma^\vee$ corresponds to the endomorphism algebra structure of $\Hom_{\LCoMod_\Gamma}(\Gamma,\Gamma)$. With notions and notations in Remarks \ref{rem:Hopf-algebroids-and-stacks} and \ref{rem:comodules-and-quasi-coherent-sheaves}, let $\mathcal{O}_{\mathfrak{X}}$ denote the structure sheaf of the stack $\mathfrak{X}$ associated to $(A,\Gamma)$, then
    $$
    \Gamma^\vee\simeq \Hom_{\LCoMod_\Gamma}(\Gamma,\Gamma)\simeq \Hom_{\QCoh(\mathfrak{X})}(e_*e^*\mathcal{O}_{\mathfrak{X}},e_*e^*\mathcal{O}_{\mathfrak{X}}).
    $$
\end{rem}

\begin{ex}\label{ex:endomorphism-algebra}
    Let $A$ be a commutative ring, and consider the Hopf algebroid $(A,A\otimes_\Z A)$ with the canonical structure morphisms. By definition, it dual algebra is 
    $$
    \Hom_{A,\eta_L}(A\otimes_\Z A,A)\simeq \Hom_\Z(A,A),
    $$
    and it is routine to check that, the algebra structure is exactly that of the endomorphism algebra.
\end{ex}

\begin{rem}\label{rem:monadicity-and-comonadicity}
    While the Hopf algebroid $(A,\Gamma)$ actually expresses the co-monadicity of the affine cover $e:\Spec(A)\rightarrow \mathfrak{X}$, the dual algebra $\Gamma^\vee$ should express the monadicity. Heuristically, suppose that the functors $e_*,e^*,e_!,e^!$ are all provided, then a quasi-coherent sheaf on $\mathfrak{X}$ can be understood as a $e^*e_*$-comodule in $\Mod_A$, or a $e^!e_!$-module in $\Mod_A$. 
\end{rem}

\begin{ex}\label{ex:Morita}
    Let $n$ be an integer, and consider the product algebra $\Z^n$, together with the very trivial morphism $e:\Spec(\Z^n)\rightarrow \Spec(\Z)$. The identification $\Shv(\{1,2,\ldots,n\};\Mod_\Z)\simeq \Mod_{\Z^n}$ suggests that, we should have functors $e^*=\Z^n\otimes_\Z-:\Mod_\Z\rightarrow \Mod_{\Z^n}$, $e_*=e_!:\Mod_{\Z^n}\rightarrow \Mod_\Z$ which is the functor of underlying abelian groups, and $e^!=\Hom_{\Z}(\Z^n,-):\Mod_\Z\rightarrow \Mod_{\Z^n}$. In this case, we have $e^!(e_!(\Z^n))=\Hom_\Z(\Z^n,\Z^n)=M_n(\Z)$, which is also the dual algebra of the Hopf algebroid $(\Z^n,\Z^n\otimes \Z^n)$; the Morita equivalence says that $\Mod_\Z\simeq \Mod_{M_n(\Z)}$. 
\end{ex}

Let $M$ be a (right) $\Gamma$-comodule, that is, $M$ is an $A$-module equipped with a comodule structure map $\psi_M:M\rightarrow M\otimes_{A,\eta_L}\Gamma$. Then the dual algebra $\Gamma^\vee$ acts on $M$: for any $D:\Gamma\rightarrow A$, $D$ acts on $M$ through the following composite 
$$
M\xrightarrow{\psi_M}M\otimes_{A,\eta_L}\Gamma \xrightarrow{\id_M\otimes D} M.
$$

\begin{rem}\label{rem:adjunctions}
    The canonical map $\Gamma\rightarrow (\Gamma^\vee)^\vee$ induces the following composite
    $$
    M\xrightarrow{\psi_M}M\otimes_{A,\eta_L}\Gamma\rightarrow M\otimes_A (\Gamma^\vee)^\vee,
    $$
    and by adjunction it further induces a morphism $\Gamma^\vee\rightarrow \Hom_\Z(M,M)$, which is in fact a ring homomorphism.
\end{rem}

\begin{rem}\label{rem:A-as-comodule}
    The module $A$ is naturally equipped with a right comodule structure through $\eta_R:A\rightarrow \Gamma\simeq A\otimes_{A,\eta_L}\Gamma$, and therefore there is a canonial ring homomorphism $\Gamma^\vee\rightarrow \Hom_\Z(A,A)$. 
\end{rem}

\begin{ex}\label{ex:the-prismatic-dual-as-subalgebra}
    Resume Example \ref{ex:the-prismatic-dual}. There is a canonical ring homomorphism $\Gamma^\vee\rightarrow \Hom_\Z(\Z_p[[q-1]],\Z_p[[q-1]])$, and in fact it is a monomorphism. The element $m_q\in \Gamma^\vee$, under this ring homomorphism, is sent to the ring automorphism $\Z_p[[q-1]]\xrightarrow{q\mapsto q^{1+p}}\Z_p[[q-1]]$. Also note that, if $(\Z_p[\zeta_p])_{\Prism}$ denotes the absolute prismatic site of $\Z_p[\zeta_p]$ (see \cite[Remark 4.7]{BS22}), then 
    $$
    \Aut_{(\Z_p[\zeta_p])_{\Prism}}(\Z_p[[q-1]],([p]_q))\simeq \Gal(\Q_p[\zeta_{p^\infty}]/\Q_p[\zeta_p])\simeq (1+p\Z_p)^\times.
    $$
    There is another important element $\varphi\in \Hom_\Z(\Z_p[[q-1]],\Z_p[[q-1]])$, the Frobenius of $\Z_p[[q-1]]$ sending $q$ to $q^p$; however, $\varphi$ does not lie in the image of $\Gamma^\vee\rightarrow \Hom_\Z(\Z_p[[q-1]],\Z_p[[q-1]])$.
\end{ex}

\begin{rem}\label{rem:comodules-and-modules-and-topology}
    It is expected that, by regarding a $\Gamma$-comodule as a $\Gamma^\vee$-module, no essential information would be lost. However, in order to recover comodule structures from module structures, the dual algebra $\Gamma^\vee$, even though $A$ and $\Gamma$ are discrete rings, must be treated as a topological algebra: $\Gamma^\vee=\Hom_{A,\eta_L}(\Gamma,A)$ should be equipped with the compact-open topology. A first observation is that, if $\Gamma$ is a free $A$-module through $\eta_L:A\rightarrow \Gamma$, then the canonical map $\Gamma\rightarrow \Hom_A(\Gamma^\vee,A)$ is an isomorphism, provided that $\Gamma^\vee$ is equipped with the compact-open topology and $\Hom_A(\Gamma^\vee,A)$, which is also equipped with the compact-open topology, is the group of continuous homomorphisms.
\end{rem}

\begin{thm}[Theorem \ref{thm:covariant-translation-for-profinitely-generated}]\label{thm:translation}
    Assume that, through the left unit map $\eta_L:A\rightarrow \Gamma$, $\Gamma$ is a free $A$-module with a countable basis. Then, by regarding comodules as modules, there is an equivalence of categories
    $$
    \RCoMod_\Gamma\xrightarrow{\sim} \Mod^d_{\Gamma^\vee},
    $$
    where $\RCoMod_\Gamma$ is the category of right $\Gamma$-comodules and $\Mod^d_{\Gamma^\vee}$ is the category of discrete $\Gamma^\vee$-modules (with continuous $\Gamma^\vee$-actions).
\end{thm}

\begin{ex}\label{ex:prismatic-crystals}
    Resume Example \ref{ex:the-prismatic-dual}. Since the adic completion or the formal geometry would bring extra technicalities, in this case, instead of arbitrary comodules, we only consider ``prismatic vector bundles'' (in the sense of \cite[$\S$ 7]{GLQ25}). Then the proof of the aforementioned theorem, perhaps with slight modifications concerning adic completion, implies that, a prismatic vector bundle is a $\Gamma^\vee$-module whose underlying $\Z_p[[q-1]]$-module is finitely generated and projective; in view of Example \ref{ex:the-prismatic-dual}, a prismatic vector bundle is thus a finitely generated projective $\Z_p[[q-1]]$-module $M$ equipped with an operator $\nabla_q:M\rightarrow M$, which satisfies the relation that $\nabla_q\circ q=q^{1+p}\circ \nabla_q+[p]_q$ and is in some sense nilpotent (see \cite[Definition 6.14, Theorem 7.12]{GLQ25}).  
\end{ex}

\begin{rem}\label{rem:symmetric-monoidal-structure}
    The category $\RCoMod_\Gamma$ is symmetric monoidal, and under the equivalence of categories $\RCoMod_\Gamma\simeq \Mod^d_{\Gamma^\vee}$, this symmetric monoidal structure should be reflected to certain co-multiplcation of the algebra $\Gamma^\vee$. But such a co-multiplication would require a notion of tensor product of general topological abelian groups, for example, the solid tensor of solid abelian groups.
\end{rem}

\begin{rem}\label{rem:functoriality}
    Let $f=(f_o,f_m):(A,\Gamma)\rightarrow (B,\Sigma)$ be a morphism of Hopf algebroids. There is the functor $f^*:\RCoMod_\Gamma\rightarrow \RCoMod_\Sigma$, which, in view of Remark \ref{rem:comodules-and-quasi-coherent-sheaves}, corresponds to the pullback of quasi-coherent sheaves on the associated stacks. However, this functoriality is not immediate from the perspective of dual algebras: there is no immediate functoriality for the passage from Hopf algebroids to their dual algebras. Yet the morphism $f$ factors as 
    $$
    (A,\Gamma)\rightarrow (B,B\otimes_{A,\eta_L}\Gamma\otimes_{\eta_R,A}B)\rightarrow (B,\Gamma),
    $$
    and sometimes (see \cite[Theorem 13]{Naumann07}) the morphism $(A,\Gamma)\rightarrow (B,B\otimes_{f_o,A,\eta_L}\Gamma\otimes_{\eta_R,A,f_o}B)$ would induce an equivalence of the categories of comodules, and for $(B,B\otimes_{A,\eta_L}\Gamma\otimes_{\eta_R,A}B)\rightarrow (B,\Gamma)$ there can be an induced morphism of dual algebras.
\end{rem}

\begin{ex}\label{ex:Frobenius}
    Resume Example \ref{ex:the-prismatic-dual}. The Hopf algebroid $(A,\Gamma)=(\Z_p[[q-1]],\Z_p[[q_0-1,q_1-1]]\{\frac{q_1-q_0}{[p]_{q_0}}\}_\delta)$ admits a Frobenius endomorphism 
    $$
    \varphi:(A,\Gamma)\rightarrow (A,\Gamma), 
    $$
    where $\varphi(q)=q^p$, $\varphi(q_0)=q_o^p$, and $\varphi(q_1)=q_1^p$. As in Remark \ref{rem:functoriality}, write $\Gamma_\varphi:=A\otimes_{\varphi,A,\eta_L}\Gamma\otimes_{\eta_R,A,\varphi}A$, then $\varphi$ factors as 
    $$
    (A,\Gamma)\xrightarrow{(\varphi,\varphi\otimes\id_\Gamma\otimes\varphi)} (A,\Gamma_\varphi)\xrightarrow{(\id_A,\widetilde{\varphi})} (A,\Gamma),
    $$
    and up to $p$-adic completion, $(\varphi,\varphi\otimes\id\otimes\varphi):(A,\Gamma)\rightarrow (A,\Gamma_\varphi)$ would induce an equivalence of the categories of quasi-coherent sheaves. Note that 
    $$
    \Gamma_\varphi=\Z_p[[q_0-1,q_1-1]]\{\frac{q_1^p-q_0^p}{[p]_{q_0^p}}\}_\delta, 
    $$
    and under this identification, $\widetilde{\varphi}$ is actually the inclusion. Similarly to Example \ref{ex:the-prismatic-dual}, $\Gamma^\vee$ is generated by an operator $\nabla_{q}^\varphi:=\frac{m_q-\id}{q(q^p-1)}$, and the morphism $\Gamma^\vee\rightarrow \Gamma_\varphi^\vee$, obviously, sends $\nabla_q\in \Gamma^\vee$ to $[p]_q\nabla_q^\varphi\in \Gamma_\varphi^\vee$.
\end{ex}

It is expected that, under the equivalence of categories $\RCoMod_\Gamma\simeq \Mod^d_{\Gamma^\vee}$, we could transport the cohomological algebra in $\RCoMod_\Gamma$ to a category of $\Gamma^\vee$-modules where we can use projective resolutions. However, since a non-trivial topology is involved, there is no immediate actualization of the expected category of $\Gamma^\vee$-modules. 

\begin{ex}\label{ex:prismatic-cohomology-through-dual}
    According to Example \ref{ex:the-prismatic-dual}, there should be a resolution of $A$ as a left $\Gamma^\vee$-modules as follows:
    $$
    0\rightarrow \Gamma^\vee\xrightarrow{-\cdot \nabla_q}\Gamma^\vee\xrightarrow{\eta_L^\vee}A\rightarrow 0,
    $$
    suggesting that, $\Ext_\Gamma(A,A)\simeq \Ext_{\Gamma^\vee}(A,A)$ should be computed by the two-term complex
    $$
    A\xrightarrow{\nabla_q}A.
    $$
    This complex is indeed correct, see for example \cite[$\S$7]{GLQ25}, or \cite[$\S$4.8]{BL22}. However, the aforementioned argument could be valid only within an appropriate category of $\Gamma^\vee$-modules larger then the category of discrete $\Gamma^\vee$-modules.
    
\end{ex}

\begin{rem}\label{rem:profinite-rings}
    Suppose that, in the Hopf algebroid $(A,\Gamma)$, $A$ is a finite ring, then $\Gamma^\vee$ is a profinite ring, for which the category of profinite $\Gamma^\vee$-modules is available, and it is indeed the case that, $\Ext_{\Gamma^\vee}(M,N)$ can be computed by either projective resolutions of $M$ or injective resolutions of $N$, provided that $M$ is a profinite $\Gamma^\vee$-module and $N$ is a discrete $\Gamma^\vee$-module. See for example \cite[$\S$5]{RZ10}.  
\end{rem}

\begin{rem}\label{rem:solid-modules}
    For a profinite ring $\Lambda$, the category of profinite $\Lambda$-modules and the category of discrete $\Lambda$-modules can both be embedded into the category of solid $\Lambda$-modules, with their cohomological algebras and tensor products preserved, see \cite{Tang24}. This, together with Remark \ref{rem:symmetric-monoidal-structure}, suggests that, the dual algebras of Hopf algebroids should be developed with the theory of solid abelian groups.  
\end{rem}

\begin{rem}\label{rem:duality}
    Since the dual algebra $\Gamma^\vee$, in general, is not commutative, the notion of left $\Gamma^\vee$-modules and that of right $\Gamma^\vee$-modules are different. However, in some cases the left modules and right modules are related. For example, if $\Lambda$ is a profinite ring, then the Pontryagin duality (see for example \cite[$\S$5.1]{RZ10}) induces an contravariant equivalence between the category of profinite left $\Lambda$-modules and the category of discrete right $\Lambda$-modules. Also, the six functor formalism (see for example \cite[$\S$4]{KNP24}) indicates that, the categories $\Mod_{\Gamma^\vee}$ and $\Mod_{\Gamma^{\vee}_\op}$ should be related with the dualizing object $\RHom_{\Gamma^\vee}(A,\Gamma^\vee)$, which again should make sense provided the theory of solid abelian groups. 
\end{rem}

It would be helpful to end this introduction by listing the basic structure of the algebra $\Gamma^\vee$.
    \begin{itemize}
        \item The multiplicative unit of $\Gamma^\vee=\Hom_A(\Gamma,A)$ is given by the co-unit $\epsilon:\Gamma\rightarrow A$ of the Hopf algebroid $(A,\Gamma)$. 

        \item Also, the co-unit $\epsilon$ induces a morphism of algebras $\epsilon^\vee:A\rightarrow \Gamma^\vee$, and $\Gamma^\vee$ is thus an $A$-algebra. However, in general $\epsilon^\vee(A)$ is not in the center of $\Gamma^\vee$: $\epsilon^\vee(A)$ is in the center of $\Gamma^\vee$ if and only if $\eta_L=\eta_R$.

        \item The $A$-module structure on $\Gamma^\vee$ from $\eta_L$ is the left action of $A$ on $\Gamma^\vee$ through $\epsilon^\vee$, and the $A$-module structure on $\Gamma^\vee$ from $\eta_R$ is the right action of $A$ on $\Gamma^\vee$ through $\epsilon^\vee$.

        \item The left unit $\eta_L:A\rightarrow \Gamma$ induces a morphism of $A$-modules $\eta_L^\vee:\Gamma^\vee\rightarrow A$. This morphism is generally not $A$-linear with respect to the right $A$-module structure on $\Gamma^\vee$.

        \item There is a canonical continuous morphism of $A$-algebras $\Gamma^\vee\rightarrow \Hom_{\Z}(A,A)$. In particular, $A$ is a left $\Gamma^\vee$-module. Moreover, $\eta_L^\vee:\Gamma^\vee\rightarrow A$ is a morphism of left $\Gamma^\vee$-modules. 

        We might call $\Gamma^\vee$ (or the Hopf algebroid $(A,\Gamma)$) restrictive if the canonical morphism $\Gamma^\vee\rightarrow \Hom_\Z(A,A)$ is a monomorphism, and in this case $\Gamma^\vee$ is a suitable algebra between $A=\Hom_A(A,A)$ and $\Hom_{\Z}(A,A)$.
    \end{itemize}

\subsection*{Acknowledgements}
The writer is very grateful to Ruochuan Liu and Guozhen Wang for suggesting the idea of dual algebras of Hopf algebroids. The writer owes a special debt of gratitude to Chris Brav for explaining subtlies concerning topological abelian groups. A special thanks to Yingdi Qin for helpful conversations. Thanks also to Rixing Fang, Hui Gao, Yuwen Gu, Zhengping Gui, Guchuan Li, Weinan Lin, Xiangrui Shen, Yupeng Wang, Jiaheng Zhao, and Shuhan Zheng for valuable discussions and suggestions.

\section{The Dual Algebras of Hopf algebroids}\label{section:the-dual-algebras}

In this section we elaborate the ideas suggested in Section \ref{section:introduction}. 

\begin{convention}\label{convention:the-dual-Hopf-algebroids}
    In this section, let $(A,\Gamma)$ be a Hopf algebroid where $A$ is a commutative ring and $\Gamma$, through the left unit map $\eta_L:A\rightarrow \Gamma$, is a free $A$-module with a countable basis, and we will refer to such a Hopf algebroid as a free Hopf algebroid. Let $\eta_R:A\rightarrow \Gamma$ denote the right unit map, $\Delta:\Gamma\rightarrow \Gamma\otimes_{\eta_R,A,\eta_L}\Gamma$ the co-multiplication, and $\epsilon:\Gamma\rightarrow A$ the co-unit, and $c:\Gamma\rightarrow \Gamma$ the conjugation. 
    
    By a $\Gamma$-comodule it means a right $\Gamma$-comodule, that is, an $A$-module $M$ equipped with a map 
    $$
    \psi_M:M\rightarrow M\otimes_{A,\eta_L}\Gamma,
    $$
    which is $A$-linear with respect to the $\eta_R$-linear structure of $\Gamma$, and is co-associative and co-unitary; let $\RCoMod_\Gamma$ denote the category of right $\Gamma$-comodules.

    Finally, almost all rings and modules considered are equipped with a topology, although sometimes the discrete topology; and $\Hom(-,-)$ is taken to be the group of continuous homomorphisms, and it is further equipped with the compact-open topology unless otherwise specified. For the group of not necessarily continuous homomorphisms, we use the notation $\Hom^{un}(-,-)$. 
\end{convention}

We would like to study the linear dual $\Hom_{A,\eta_L}(\Gamma,A)$, which turns out to be a topological ring and encodes almost all of the essential information associated to the Hopf algebroid $(A,\Gamma)$.

\begin{defn}\label{defn:profinite-dual}
    The dual algebra of the Hopf algebroid $(A,\Gamma)$ is defined to be 
    $$
    \Gamma^\vee:=\Hom_{A,\eta_L}(\Gamma,A),
    $$
    where the algebra structure is determined as follows: for $D_1,D_2\in \Gamma^\vee$, $D_2\circ D_1\in \Gamma^\vee$ is the composite 
    $$
    D_2\circ D_1:\Gamma \xrightarrow{\Delta}\Gamma\otimes_{\eta_R,A\eta_L}\Gamma\xrightarrow{\id_\Gamma\otimes D_1}\Gamma\xrightarrow{D_2} A,
    $$
    and the unit of $\Gamma^\vee$ is the co-unit of the Hopf algebroid $\epsilon:\Gamma\rightarrow A$.
    
    Moreover, $\Gamma^\vee=\Hom_{A,\eta_L}(\Gamma,A)$ is equipped with the compact-open topology where $A$ is regarded as a discrete topological ring, and $\Gamma,A$ are discrete topological $A$-modules.
\end{defn}

\begin{lem}\label{lem:dual-algebra}
    The dual algebra $\Gamma^\vee$, defined in Definition \ref{defn:profinite-dual}, is indeed an algebra.
\end{lem}

\begin{proof}
    It is immediate to see that the multiplication described in Definition \ref{defn:profinite-dual} is bi-additive. 
    
    To check that it is associative, it suffices to resort to the co-associativity of the Hopf algebroid, that is, we have the following commutative diagram
    $$
    \xymatrix@R=50pt@C=50pt{
    \Gamma \ar[r]^-\Delta \ar[d]^\Delta & \Gamma\otimes_{\eta_R,A,\eta_L}\Gamma \ar[d]^{\id_\Gamma\otimes \Delta}\\
    \Gamma\otimes_{\eta_R,A,\eta_L}\Gamma \ar[r]^-{\Delta\otimes\id_\Gamma} & \Gamma\otimes_{\eta_R,A,\eta_L}\Gamma\otimes_{\eta_R,A,\eta_L}\Gamma
    },
    $$
    and for any $D_1,D_2,D_3\in \Gamma^\vee$, consider the following composite 
    $$
    D_{321}:\Gamma\otimes_{\eta_R,A,\eta_L}\Gamma\otimes_{\eta_R,A,\eta_L}\Gamma
    \xrightarrow{\id_{\Gamma\otimes_{\eta_R,A,\eta_L}\Gamma\otimes D_1}}\Gamma\otimes_{\eta_R,A,\eta_L}\Gamma\xrightarrow{\id_\Gamma\otimes D_2}\Gamma\xrightarrow{D_3}A,
    $$
    then 
    $$
    D_3\circ (D_2\circ D_1)=D_{321}\circ (\id_\Gamma\otimes\Delta)\circ \Delta=D_{321}\circ (\Delta\otimes \id_\Gamma)\circ \Delta=(D_3\circ D_2)\circ D_1.
    $$

    It remains to check that the multiplication is unitary with unit $\epsilon:\Gamma\rightarrow A$. For any $D\in \Gamma^\vee$, $D\circ \epsilon$ is the following composite
    $$
    \Gamma\xrightarrow{\Delta}\Gamma\otimes_{\eta_R,A,\eta_L}\Gamma\xrightarrow{\id_\Gamma\otimes \epsilon}\Gamma\xrightarrow{D}A,
    $$
    and since the Hopf algebroid is co-unitary, this composite is again $D:\Gamma\rightarrow A$; $\epsilon \circ D$ is the following composite 
    $$
    \Gamma\xrightarrow{\Delta}\Gamma\otimes_{\eta_R,A,\eta_L}\Gamma\xrightarrow{\id_\Gamma\otimes D}\Gamma\xrightarrow{\epsilon}A,
    $$
    and note that $\epsilon:\Gamma\rightarrow A$ is also $\eta_R$-linear, this composite coincides with 
    $$
    \Gamma\xrightarrow{\Delta}\Gamma\otimes_{\eta_R,A,\eta_L}\Gamma\xrightarrow{\epsilon\otimes \id_\Gamma}\Gamma \xrightarrow{D}A,
    $$
    and is again $D:\Gamma\rightarrow A$.
\end{proof}

\begin{lem}\label{lem:profinite-dual-topological-algebra}
    The dual algebra $\Gamma^\vee$, defined in Definition \ref{defn:profinite-dual}, is indeed a topological ring.
\end{lem}

\begin{proof}
    By definition, since both $\Gamma$ and $A$ is equipped with the discrete topology, 
    the compact-open topology of $\Gamma^\vee$ is generated by subsets of the form $B(K,U)=\{D\in \Gamma^\vee|D(K)\subset U\}$, where $K$ is a finite subset of $\Gamma$, and $U$ is a subset of $A$. Therefore, it suffices to consider finite intersections of sets of the form $B(\gamma,a)=\{D\in \Gamma^\vee|D(\gamma)=\{a\}\}$, where $\gamma\in \Gamma$, $a\in A$.



    First we check that the addition of $\Gamma^\vee=\Hom_{A,\eta_L}(\Gamma,A)$ is continuous. Given $D_1,D_2\in \Gamma^\vee$ with $D_1+D_2\in B(\gamma,a)$ (strictly speaking we shall consider a finite intersection of sets of the form $B(\gamma,a)$, but the argument is essentially the same, up to a possibly finite intersection), where $\gamma\in \Gamma$ and $a\in A$, in particular $a=D_1(\gamma)+D_2(\gamma)$. Then $B(\gamma,D_1(\gamma))$ is a neighborhood of $D_1$, and $B(\gamma,D_2(\gamma))$ is a neighborhood of $D_2$, and $B(\gamma,D_1(\gamma))+B(\gamma,D_2(\Gamma))\subset B(\gamma,a)$. 
    

    Then we check that the multiplication of $\Gamma^\vee$ is continuous. Given $D_1,D_2\in \Gamma^\vee$ with $D_2\circ D_1\in B(\gamma,a)$, where $\gamma\subset \Gamma$ and $a\in A$. Recall that $D_2\circ D_1$ is the following composite 
    $$
    \Gamma\xrightarrow{\Delta}\Gamma\otimes_{\eta_R,A,\eta_L}\Gamma\xrightarrow{\id_\Gamma \otimes D_1}\Gamma \xrightarrow{D_2}A;
    $$
    write $\Delta(\gamma)=\sum_{1\leq i\leq k}\gamma_i'\otimes\gamma_i\in \Gamma\otimes_{\eta_R,A,\eta_L}\Gamma$, then 
    $$
    (\id_\Gamma \otimes D_1)(\Delta(\Gamma))= \sum_{1\leq i\leq k}\gamma_i'\eta_R(D_1(\gamma_i)):=\widetilde{\gamma},
    $$
    and $D_2(\widetilde{\gamma})=a$. Thus $\cap_{1\leq i\leq k}B(\gamma_i,D_1(\gamma_i))$ is a neighborhood of $D_1$, and $B(\widetilde{\gamma},a)$ is a neighborhood of $D_2$, and $B(\widetilde{\gamma},a)\circ (\cap_{1\leq i\leq k}B(\gamma_i,D_1(\gamma_i)))\subset B(\gamma,a)$.
\end{proof}




\begin{lem}\label{lem:dual-of-free-elaborated}
    Let $M$ be a free $A$-module, and choose a basis $\{e_i\}_{i\in I}$ of $M$ as a free $A$-module. Then the compact-open topology of $M^\vee:=\Hom_A(M,A)$ coincides with the product topology of $M^\vee\simeq \prod_{i\in I}Ae_i^\vee$, where each $Ae_i^\vee$ is equipped with the discrete topology.
\end{lem}

\begin{proof}
    Since $M$ and $A$ are discrete, the compact-open topology of $M^\vee$ is generated by subsets of the form $B(m,a)=\{f\in \Hom_A(M,A)|f(m)=a\}$, where $m\in M$, $a\in A$. Similarly, the product topology of $M^\vee\simeq \prod_{i\in I}Ae_i^\vee$ is generated by subsets of the form $\{a_ie_i^\vee\}\times\prod_{j\neq i}Ae_j^\vee$. It is immediate to see that, $\{a_ie_i^\vee\}\times\prod_{j\neq i}Ae_j^\vee=B(e_i,a)$; also, if $m=\sum_{i}m_ie_i$, then $B(m,a)$ is a union of subsets of the form $\cap_{i,m_i\neq 0}B(e_i,a_i)$ where $\sum_i m_ia_i=a$. Therefore the compact-open topology of $M^\vee$ coincides with the product topology of $M^\vee\simeq \prod_{i\in I}Ae_i^\vee$.   
\end{proof}

\begin{ex}\label{ex:dual-of-divided-power}
    Consider the Hopf algebra ($\F_p$,$\F_p\langle t\rangle$), where the underlying ring of $\F_p\langle t\rangle$ is the divided power polynomial over $\F_p$ generated by the indeterminate $t$, and the underlying co-multiplication is given by the ring homomorphism $\F_p\langle t\rangle\xrightarrow{t\mapsto t\otimes 1+1\otimes t}\F_p\langle t\rangle\otimes_{\F_p}\F_p\langle t\rangle$. 
    
    Then $\Hom_{\F_p}(\F_p\langle t\rangle,\F_p)\simeq \F_p[[t^\vee]]$, whose underlying ring is the formal power series ring generated by $t^\vee$ which is the dual basis of $t$, and it is also equipped with the co-multiplication $\F_p[[t^\vee]]\xrightarrow{t^\vee\mapsto t^\vee\otimes 1+1\otimes t^\vee} \F_p[[t^\vee]]\widehat{\otimes}_{\F_p}\F_p[[t^\vee]]$.

    The underlying ring of $\F_p[[t^\vee]]\simeq \prod_{\N}\F_p$, if equipped with the product topology where $\F_p$ is discrete, is a profinite ring. Further, if by $\Hom_{\F_p}(\F_p[[t^\vee]],\F_p)$ it means the group of continuous $\F_p$-homomorphisms from $\F_p[[t^\vee]]$ to $\F_p$, we have $\Hom_{\F_p}(\F_p[[t^\vee]],\F_p)\simeq \F_p\langle t\rangle$.


\end{ex}


\begin{rem}\label{rem:profinite-dual-as-an-A-algebra}
    By taking the $\eta_L$-linear dual of the co-unit $\epsilon:\Gamma\rightarrow A$, there is a map $\epsilon^\vee:A\rightarrow \Gamma^\vee$, such that for $a\in A$, $\epsilon^\vee(a)=a\epsilon \in \Gamma^\vee$, and it is immediate to see that $\epsilon^\vee:A\rightarrow \Gamma^\vee$ is a ring homomorphism. However, in general $\epsilon^\vee(A)$ does not lie in the center of $\Gamma^\vee$, and $\Gamma^\vee$ admits two $A$-module structures:
    $$
    (a,D)\mapsto aD\quad\text{or}\quad D\circ (a\epsilon),\quad \text{where $a\in A$, $D\in \Gamma^\vee$},
    $$
    we will refer to the first as the left $A$-module structure of $\Gamma^\vee$, and the latter as the right $A$-module structure of $\Gamma^\vee$. 
\end{rem}

\begin{rem}\label{rem:right-module-structure-for-profinite-dual}
    The right unit map $\eta_R:A\rightarrow \Gamma$, making $\Gamma$ also an $A$-module, induces an $A$-module structure on $\Gamma^\vee=\Hom_{A,\eta_L}(\Gamma,A)$, such that 
    $$
    (a,D)\mapsto D(\eta_R(a)\cdot -),\quad \text{where $a\in A$, $D\in \Gamma^\vee$},
    $$
    and it is immediate to check, by definition of the multiplication of $\Gamma^\vee$, that $D(\eta_R(a)\cdot-):\Gamma\rightarrow A$ is the composite
    $$
    D\circ(a\epsilon):\Gamma\xrightarrow{\Delta}\Gamma\otimes_{\eta_R,A,\eta_L}\Gamma\xrightarrow{\id_\Gamma\otimes a\epsilon}\Gamma\xrightarrow{D}A,
    $$
    and therefore this $A$-module structure on $\Gamma^\vee$ induced by $\eta_R$ is exactly the right $A$-module structure introduced in Remark \ref{rem:profinite-dual-as-an-A-algebra}.
\end{rem}

\begin{rem}\label{rem:conjugation-and-dual}
    The conjugation $c:\Gamma\rightarrow \Gamma$ induces an isomorphism (which is abusively still denoted as $c$)
    $$
    c:\Hom_{A,\eta_L}(\Gamma,A)\xrightarrow{\sim} \Hom_{A,\eta_R}(\Gamma,A),
    $$
    which suggests that the formation of the dual algebra $\Gamma^\vee$ is essentially independent of the choice of $\eta_L$ or $\eta_R$. 
\end{rem}

\begin{rem}\label{rem:profinite-dual-to-A}
    By taking the $\eta_L$-linear dual of the left unit $\eta_L:A\rightarrow \Gamma$, there is a map $\eta_L^\vee:\Gamma^\vee\rightarrow A$, such that for $D:\Gamma\rightarrow A$, $\eta_L^\vee(D)=D(1)$. It is obvious that $\eta_L^\vee:\Gamma^\vee\rightarrow A$ is linear with respect to the left $A$-module structure of $\Gamma^\vee$, but in general it is not linear with respect to the right $A$-module structure on $\Gamma^\vee$.
\end{rem}

\begin{rem}\label{rem:A-as-profinite-dual-module}
    The right unit map $\eta_R:A\rightarrow \Gamma$ gives rise to an action of $\Gamma^\vee$ on $A$: for $D\in \Gamma^\vee=\Hom_{A,\eta_L}(\Gamma,A)$, $D$ acts on $A$ through the following composite
    $$
    A\xrightarrow{\eta_R}\Gamma\xrightarrow{D}A.
    $$
    In fact, through this action $A$ will be regarded as a discrete $\Gamma^\vee$-module, that is, it corresponds to a continuous ring homomorphism $\Gamma^\vee\rightarrow \Hom(A,A)$, which is moreover a homomorphism of $A$-algebras, where $\Gamma^\vee$ is regarded as an $A$-algebra through $\epsilon^\vee$, and $\Hom(A,A)$ is regarded as an $A$-algebra through the multiplication of $A$. 

    In general, a (right) $\Gamma$-comodule can be naturally regarded as a (left) discrete $\Gamma^\vee$-module. Roughly, let $M$ be a right $\Gamma$-comodule with the comodule structure map $\psi_M:M\rightarrow M\otimes_{A,\eta_L}\Gamma$, then for $D\in \Gamma^\vee$, $D$ acts on $M$ through the following composite
    $$
    M\xrightarrow{\psi_M}M\otimes_{A,\eta_L}\Gamma\xrightarrow{D}M.
    $$
    Details will come in the sequel.
\end{rem}

\begin{defn}\label{defn:modules-over-profinite-dual}
    Let $\Mod^{d}_{\Gamma^\vee}$ denote the category of (left) discrete $\Gamma^\vee$-modules with continuous $\Gamma^\vee$-actions, that is, 
    \begin{itemize}
        \item an object of $\Mod^d_{\Gamma^\vee}$ is a discrete abelian group $M$ together with a continuous ring homomorphism $\Gamma^\vee\rightarrow \Hom_\Z(M,M)$, where $\Hom_\Z(M,M)$ is equipped with the compact-open topology;

        \item a morphism between two such objects $M$ and $N$ is a morphism of abelian groups $f:M\rightarrow N$ such that the following diagram is commutative:
        $$
        \xymatrix@R=50pt@C=50pt{
        \Gamma^\vee \ar[r] \ar[d] & \Hom_\Z(M,M) \ar[d]^{f\circ -} \\
        \Hom_\Z(N,N) \ar[r]^{-\circ f}       & \Hom_\Z(M,N).
        }
        $$
    \end{itemize}

    Similarly, let $\Mod^d_{\Gamma^{\vee}_\op}$ denote the category of right discrete $\Gamma^\vee$-modules with continuous $\Gamma^\vee$-actions, where $\Gamma^\vee_{\op}$ denotes the opposite algebra of $\Gamma^\vee$.
\end{defn}

\begin{construction}\label{construction:covariant-translation-for-profinitely-generated}
    Let $M$ be a right $\Gamma$-comodule with the comodule structure map 
    $$
    \psi_M:M\rightarrow M\otimes_{A,\eta_L}\Gamma.
    $$
    
    There is a canonical isomorphism 
    $$
    M\otimes_{A,\eta_L}\Gamma\simeq \Hom_A(\Gamma^\vee,M),
    $$
    where $\Hom_A(\Gamma^\vee,M)$ is the group of continuous $A$-homomorphisms of topological $A$-modules, and $\Gamma^\vee$ by definition is equipped with the compact-open topology, and $M$ is equipped with the discrete topology.

    Further, the map $\psi_M:M\rightarrow M\otimes_{A,\eta_L}\Gamma\simeq \Hom_A(\Gamma^\vee,M)$, by adjunction, corresponds to a map 
    $$
    \widetilde{\psi}_M:\Gamma^\vee\rightarrow \Hom(M,M), 
    $$
    which is in fact a continuous homomorphism of topological rings, where $\Hom(M,M)$ is equipped with the compact-open topology. 
    
    This construction will be justified in the sequel. In fact, through this construction we will have an equivalence of categories 
    $$
    \RCoMod_{\Gamma}\simeq \Mod^d_{\Gamma^\vee}.
    $$
\end{construction}

\begin{lem}\label{lem:tensor-Hom-profinitely-generated}
    Let $N$ be a free $A$-module, $N'$ an $A$-module, then the following canonical morphism is an isomorphism:
    $$
    N\otimes_A N'\xrightarrow{\sim} \Hom_A(N^\vee,N'),
    $$
    where $N^\vee=\Hom_A(N,A)$, and $\Hom_A(N^\vee,N')$ is the group of continuous $A$-homomorphisms of topological $A$-modules, and $N^\vee$ is equipped with the compact-open topology, and $N'$ the discrete topology.

    Note that (see Remark \ref{rem:tensor-Hom-profinitely-generated-topology}), the compact-open topology for $\Hom_A(N^\vee,N')$ is discrete. In particular, take $N'=A$, there is the canonical isomorphism of topological $A$-modules 
    $$
    N\xrightarrow{\sim}\Hom_A(N^\vee,A),
    $$
    where, of course, $\Hom_A(N^\vee,A)$ denotes the group of continuous $A$-homomorphisms.
\end{lem}

\begin{proof}   
    Choose a basis $\{e_i\}_{i\in I}$ of $N$ as a free $A$-module. Then $N\simeq \oplus_{i\in I}Ae_i$, and $N^\vee\simeq \prod_{i\in I}Ae_i^\vee$. Let $f:N^\vee\simeq \prod_{i\in I}Ae_i^\vee\rightarrow N'$ be a continuous $A$-homomorphism, then $\Ker(f)$ is a neighborhood of $0$ in $N^\vee$, and therefore $\Ker(f)$ contains a subset of the form $\prod_{i\in I}U_ie_i^\vee$ where $0\in U_i\subset A$ and all but finitely many $U_i=A$; therefore and the canonical morphism $N\otimes_A N'\rightarrow\Hom_A(N^\vee,N')$ is an epimorphism. Also, under the identification 
    $N\otimes_A N'\simeq \oplus_{i\in I}(Ae_i\otimes N')$, the canonical morphism $N\otimes_A N'\rightarrow \Hom_A(N^\vee,N')$ is given by the assignment
    $$
    e_i\otimes n'\mapsto (e_i\otimes n':N^\vee\rightarrow N',(e_i\otimes n')(\{a_je_j^\vee\}_{j\in I})=a_in'),
    $$
    and it is immediate to see that it is a monomorphism.
\end{proof}

\begin{rem}\label{rem:tensor-Hom-profinitely-generated-topology}
    In Lemma \ref{lem:tensor-Hom-profinitely-generated}, the isomorphism $N\otimes_A N'\xrightarrow{\sim}\Hom_A(N^\vee,N')$ is in fact an isomorphism of topological $A$-modules; in other words, the compact-open topology of $\Hom_A(N^\vee,N')$ is discrete: choose a basis $\{e_i\}_{i\in I}$ of $N$ as an $A$-module, then a continuous morphism of $A$-modules $f:N^\vee\rightarrow N'$ is determined by its restriction a finite subset of $\{e_i^\vee\}\subseteq \prod_{i\in i}Ae_i^\vee\simeq N^\vee$, where $e_i^\vee$ is the dual basis of $e_i$, thus it is immediate to check by definition that the compact-open topology for $\Hom_A(N^\vee,A)$ is discrete.
\end{rem}

\begin{rem}\label{rem:tensor-Hom-profinitely-generated}
    Lemma \ref{lem:tensor-Hom-profinitely-generated} will be applied to the free $A$-module $N=\Gamma$, in particular, there is the canonical isomorphism of $A$-modules 
    $$
    \Gamma\xrightarrow{\sim}\Hom_A(\Gamma^\vee,A).
    $$
    Note that, $\Gamma^\vee$ is in fact an $A$-$A$-bimodule, namely, it has the left and right $A$-module structures as introduced in Remark \ref{rem:profinite-dual-as-an-A-algebra}; therefore, $\Hom_{A,left}(\Gamma^\vee,A)$ admits an $A$-module structure induced by the right $A$-module structure of $\Gamma^\vee$; according to Remark \ref{rem:right-module-structure-for-profinite-dual}, the $A,\eta_R$-module structure of $\Gamma$, through this canonical isomorphism, corresponds to the $A$-module structure on $\Hom_{A,left}(\Gamma^\vee,A)$ induced by the right $A$-module structure of $\Gamma^\vee$. 
\end{rem}

\begin{cor}\label{cor:tensor-Hom-profinitely-generated}
    Let $N'$ be an $A$-module. The canonical morphism of $A$-modules
    $$
    N'\otimes_{A,\eta_L}\Gamma\rightarrow \Hom_{A,left}(\Gamma^\vee,N')
    $$
    is an isomorphism, and is in fact an isomorphism of $A$-$A$-bimodules, where the extra $A$-module structure of $N'\otimes_{A,\eta_L}\Gamma$ is induced by the $A,\eta_R$-module structure of $\Gamma$, an the extra $A$-module structure of $\Hom_{A,left}(\Gamma^\vee,N')$ is induced by the right $A$-module structure of $\Gamma^\vee$.
\end{cor}

\begin{proof}
    Combine Lemma \ref{lem:tensor-Hom-profinitely-generated} and Remark \ref{rem:tensor-Hom-profinitely-generated}.
\end{proof}

\begin{lem}\label{lem:tensor-Hom-adjunction-profinite}
    Let $M,M'$ be $A$-modules. Then there is a canonical isomorphism of abelian groups 
    $$
    \Hom_{A,right}(M,\Hom_{A,left}(\Gamma^\vee,M'))\simeq \Hom_{A-A}(\Gamma^\vee,\Hom_\Z(M,M')),
    $$
    where $\Hom_{A-A}(-,-)$ denotes the group of continuous $A$-$A$-bimodule homomorphisms, and the $A$-$A$-bimodule structure of $\Hom_{\Z}(M,M')$ is induced by the $A$-module structures of $M,M'$.

    In fact, the following canonical monomorphisms 
    $$
    \Hom_{A,right}(M,\Hom_{A,left}(\Gamma^\vee,M'))\rightarrow \Hom^{un}_{A,left}(M\otimes_{A,right}\Gamma^\vee,M'),
    $$
    $$
    \Hom_{A-A}(\Gamma^\vee,\Hom_\Z(M,M'))\rightarrow \Hom_{A,left}^{un}(M\otimes_{A,right}\Gamma^\vee,M'),
    $$
    have the same image.
\end{lem}

\begin{proof}
    We first treat the special case where $M=A$, and it amounts to establish a canonical isomorphism 
    $$
    \Hom_{A-A}(\Gamma^\vee,\Hom_\Z(A,M'))\simeq \Hom_{A,left}(\Gamma^\vee,M').
    $$
    An element $f\in \Hom_{A-A}(\Gamma^\vee,\Hom_\Z(A,M'))$ is of the form $(f(-))(\star)$, such that for $D\in \Gamma^\vee$, $(f(D))(\star):A\rightarrow M'$ is a group homomorphism, and for $a,a'\in A$, $(f(aDa'))(\star)=a(f(D))(a'\cdot\star)$; thus there is indeed a canonical map 
    \begin{align*}
    \Hom_{A-A}(\Gamma^\vee,\Hom_\Z(A,M')) & \rightarrow & \Hom_{A,left}(\Gamma^\vee,M')\\
    f:\Gamma^\vee\rightarrow \Hom_{\Z}(A,M') & \mapsto & (f(-))(1):\Gamma^\vee\rightarrow M',
    \end{align*}
    and it is routine to check that it is an isomorphism.

    It is then clear that, if $M=\oplus_I A$, the statement holds. In general, it suffices to embed $M$ into an exact sequence of the form 
    $$
    \oplus_{I'}A\rightarrow \oplus_I A\rightarrow M\rightarrow 0.
    $$
\end{proof}

\begin{rem}\label{rem:tensor-Hom-adjunction-profinite-1}
    If the topological consideration is ignored, Lemma \ref{lem:tensor-Hom-adjunction-profinite} is the application of the following more general statement: let $M$ be an $R$-module, $M'$ be an $R'$-module, and $N$ an $R$-$R'$-bimodule, then there are the canonical isomorphisms
    \begin{align*}
    \Hom_R(M,\Hom_{R'}(N,M'))
    &\simeq \Hom_{R'}(M\otimes_R N,M') \\
    &\simeq \Hom_{R-R'}(N,\Hom_\Z(M,M')).
    \end{align*}
\end{rem}

\begin{rem}\label{rem:tensor-Hom-adjunction-profinite-2}
    Combine Corollary \ref{cor:tensor-Hom-profinitely-generated} and Lemma \ref{lem:tensor-Hom-adjunction-profinite}, it is easy to figure out that, given a morphism $f:M\rightarrow M'\otimes_{A,\eta_L}\Gamma\simeq \Hom_{A,left}(\Gamma^\vee,M')$, there is the corresponding morphism $\widetilde{f}:\Gamma^\vee\rightarrow \Hom_\Z(M,M')$, such that for $D\in \Gamma^\vee$, $\widetilde{f}(D):M\rightarrow M'$ is the following composite
    $$
    M\xrightarrow{f}M'\otimes_{A,\eta_L}\Gamma\xrightarrow{\id_{M'}\otimes D}M'.
    $$
\end{rem}

\begin{lem}\label{lem:covariant-translation-for-profinitely-generated}
    Let $M$ be an $A$-module equipped with an $A,\eta_R$-linear map
    $$
    \psi_M:M\rightarrow M\otimes_{A,\eta_L}\Gamma.
    $$
    By Corollary \ref{cor:tensor-Hom-profinitely-generated} and Lemma \ref{lem:tensor-Hom-adjunction-profinite}, $\psi_M\in \Hom_{A,\eta_R}(M,M\otimes_{A,\eta_L}\Gamma)\simeq \Hom_{A,right}(M,\Hom_{A,left}(\Gamma^\vee,M))$ corresponds to $\widetilde{\psi}_M\in \Hom_{A-A}(\Gamma^\vee,\Hom_\Z(M,M))$, and vice versa.
    
    Then $\widetilde{\psi}_M:\Gamma^\vee\rightarrow \Hom_\Z(M,M)$ is an $A$-algebra homomorphism if and only if $\psi_M:M\rightarrow M\otimes_{A,\eta_L}\Gamma$ is a comodule structure map, that is, if and only if $\psi_M$ is co-associative and co-unitary.
\end{lem}

\begin{proof}
    First we demonstrate that, $\psi_M$ is co-associative if and only if $\widetilde{\psi}_M$ preserves multiplication. Note that, for any $D\in\Gamma^\vee$ there is always the following commutative diagram
    $$
    \xymatrix@R=50pt@C=50pt{
    M\otimes_{A,\eta_L}\Gamma \ar[r]^-{\psi_M\otimes\id_\Gamma} \ar[d]^{\id_M\otimes D} & M\otimes_{A,\eta_L}\Gamma\otimes_{\eta_R,A,\eta_L}\Gamma \ar[d]^{\id_{M\otimes_{A,\eta_L}\Gamma}\otimes D} \\
    M \ar[r]^{\psi_M} & M\otimes_{A,\eta_L}\Gamma,
    }
    $$
    since the two ways to go from the upper left corner to the lower right corner are both $\psi_M\otimes D$. The co-associativity of $\psi_M$ is the commutativity of the following diagram
    $$
    \xymatrix@R=50pt@C=50pt{
    M \ar[r]^{\psi_M} \ar[d]^{\psi_M} & M\otimes_{A,\eta_L}\Gamma \ar[d]^{\id_M\otimes\Delta} \\
    M\otimes_{A,\eta_L}\Gamma \ar[r]^-{\psi_M\otimes \id_\Gamma} & M\otimes_{A,\eta_L}\Gamma\otimes_{\eta_R,A,\eta_L}\Gamma;
    }
    $$
    on the other hand, for any $D_1,D_2\in \Gamma^\vee$, $\widetilde{\psi}_M(D_2\circ D_1)$ is the following composite:
    $$
    M\xrightarrow{\psi_M}M\otimes_{A,\eta_L}\Gamma\xrightarrow{\id_M\otimes\Delta}M\otimes_{A,\eta_L}\Gamma\otimes_{\eta_R,A,\eta_L}\Gamma\xrightarrow{\id_{M\otimes_{A,\eta_L}\Gamma}\otimes D_1}M\otimes_{A,\eta_L}\Gamma\xrightarrow{\id_M\otimes D_2}M,
    $$
    and $\widetilde{\psi}_M(D_2)\circ \widetilde{\psi}_M(D_1)$ is the following composite:
    $$
    M\xrightarrow{\psi_M}M\otimes_{A,\eta_L}\Gamma\xrightarrow{\psi_M\otimes\id_\Gamma}M\otimes_{A,\eta_L}\Gamma\otimes_{\eta_R,A,\eta_L}\Gamma\xrightarrow{\id_{M\otimes_{A,\eta_L}\Gamma}\otimes D_1}M\otimes_{A,\eta_L}\Gamma\xrightarrow{\id_M\otimes D_2}M;
    $$
    and we would like to demonstrate that, $\psi_M$ is co-associative if and only if, for any $D_1,D_2\in \Gamma^\vee$, $\widetilde{\psi}_M(D_2\circ D_1)=\widetilde{\psi}_M(D_2)\circ\widetilde{\psi}_M(D_1)$. Apply Lemma \ref{lem:tensor-Hom-adjunction-profinite} iteratively, a morphism $M\rightarrow M\otimes_{A,\eta_L}\Gamma\otimes_{\eta_R,A,\eta_L}\Gamma$ corresponds to a morphism $\Gamma^\vee\rightarrow \Hom_\Z(M,M\otimes_{A,\eta_L}\Gamma)\simeq \Hom_{A,left}(\Gamma^\vee,\Hom_\Z(M,M))$; therefore, two morphisms $f,g:M\rightarrow M\otimes_{A,\eta_L}\Gamma\otimes_{\eta_R,A,\eta_L}\Gamma$ are the same if and only if for any $D_1,D_2\in \Gamma^\vee$,
    $$
    f\circ (\id_M\otimes D_2)\circ (\id_{M\otimes_{A,\eta_L}\Gamma}\otimes D_1)=g\circ (\id_M\otimes D_2)\circ (\id_{M\otimes_{A,\eta_L}\Gamma}\otimes D_1).
    $$
    As a consequence, $\psi_M$ is co-associative if and only if $\widetilde{\psi}_M$ preserves multiplication.

    Then we demonstrate that, $\psi_M$ is co-unitary if and only if $\widetilde{\psi}_M$ preserves the multiplicative identity. The co-unitarity of $\psi_M$ is the commutativity of the following diagram 
    $$
    \xymatrix@R=50pt@C=50pt{
    M \ar[rr]^{\id_M} \ar[dr]^{\psi_M} &\  & M \\
    \ & M\otimes_{A,\eta_L}\Gamma; \ar[ur]^{\id_M\otimes \epsilon} & \ 
    }
    $$
    on the other hand, the multiplicative identity of $\Gamma^\vee$ is the co-unit map $\epsilon:\Gamma\rightarrow A$, and $\widetilde{\psi}_M(\epsilon)$ is the following composite
    $$
    M\xrightarrow{\psi_M}M\otimes_{A,\eta_L}\Gamma\xrightarrow{\id_M\otimes\epsilon}M.
    $$
    It is then obvious that $\psi_M$ is co-unitary if and only if $\widetilde{\psi}_M$ preserves the multiplicative identity.

    Finally, note that $\widetilde{\psi}_M:\Gamma^\vee\rightarrow \Hom_\Z(M,M)$ is already an $A$-$A$-bimodule map, and when it is a ring homomorphism, it is automatically an $A$-algebra homomorphism.
\end{proof}

\begin{thm}\label{thm:covariant-translation-for-profinitely-generated}
    Construction \ref{construction:covariant-translation-for-profinitely-generated} gives rise to an equivalence of categories
    $$
    \RCoMod_\Gamma\xrightarrow{\sim} \Mod^d_{\Gamma^\vee},
    $$
    where $\RCoMod_\Gamma$ is the category of right $\Gamma$-comodules and $\Mod^d_{\Gamma^\vee}$ is the category of discrete $\Gamma^\vee$-modules.
\end{thm}

\begin{proof}
    It suffices to refer to Corollary \ref{cor:tensor-Hom-profinitely-generated}, Lemma \ref{lem:tensor-Hom-adjunction-profinite}, and Lemma \ref{lem:covariant-translation-for-profinitely-generated}, together with Construction \ref{construction:covariant-translation-for-profinitely-generated}. Nevertheless, we repeat the essential points.

    Given a right $\Gamma$-comodule $M$ with comodule structure map $\psi_M:M\rightarrow M\otimes_{A,\eta_L}\Gamma$. According to Corollary \ref{cor:tensor-Hom-profinitely-generated}, Lemma \ref{lem:tensor-Hom-adjunction-profinite}, and Lemma \ref{lem:covariant-translation-for-profinitely-generated}, $\psi_M:M\rightarrow M\otimes_{A,\eta_L}\Gamma\simeq \Hom_{A,left}(\Gamma^\vee,M)$ corresponds to a continuous $A$-algebra homomorphism $\widetilde{\psi}_M:\Gamma^\vee\rightarrow \Hom_\Z(M,M)$, that is, $M$ is a discrete $\Gamma^\vee$-module. It is obvious that in this way we have a functor $\RCoMod_\Gamma\rightarrow \Mod^d_{\Gamma^\vee}$.

    This functor is fully faithful: according to Lemma \ref{lem:covariant-translation-for-profinitely-generated}, the $\Gamma^\vee$-module structure corresponds exactly to the $\Gamma$-comodule structure.

    This functor is essentially surjective: let $M$ be a discrete $\Gamma$-module, then through $\epsilon^\vee:A\rightarrow \Gamma^\vee$, $M$ is regarded as an $A$-module, and again Lemma \ref{lem:covariant-translation-for-profinitely-generated} implies that $M$ is a right $\Gamma$-comodule.
\end{proof}

\section{The Functoriality}\label{section:the-functoriality}

In this section we discuss the functoriality for the construction of dual algebras of Hopf algebroids. 

\begin{rem}\label{rem:functoriality-for-comodules}
     Let $f=(f_o,f_m):(A,\Gamma)\rightarrow (B,\Sigma)$ be a morphism of Hopf algebroids. There is the functor $f^*:\RCoMod_\Gamma\rightarrow \RCoMod_{\Sigma}$, which sends a $\Gamma$-comodule $M$, with the comodule structure $M\xrightarrow{\psi_M}M\otimes_{A,\eta_L}\Gamma$, to a $\Sigma$-comodule $M\otimes_A B$, with the comodule structure map $M\otimes_A B\rightarrow (M\otimes_A B)\otimes_{B,\eta_R}\Sigma$ induced from the following composite:
     $$
     M\xrightarrow{\psi_M}M\otimes_{A,\eta_L}\Gamma\xrightarrow{\id_M\otimes g}M\otimes_{A,\eta_L}\Sigma\simeq (M\otimes_{A}B)\otimes_{B,\eta_R}\Sigma. 
     $$
\end{rem}

\begin{rem}\label{rem:equivalence-of-Hopf-algebroids}
    Let $f=(f_o,f_m):(A,\Gamma)\rightarrow (B,\Sigma)$ be a morphism of Hopf algebroids. In general this morphism does not immediately induce a morphism $\Sigma^\vee\rightarrow \Gamma^\vee$, although, with the perspective of Proposition \ref{thm:covariant-translation-for-profinitely-generated}, there should always be a functor $\Mod^d_{\Gamma^\vee}\rightarrow \Mod^d_{\Sigma^\vee}$ corresponding to $\RCoMod_\Gamma\rightarrow \RCoMod_{\Sigma}$. However, it is still possible to understand this functoriality through a morphism of dual algebras, by choosing a suitable Hopf algebroid ``equivalent to $(A,\Gamma)$''. 
    
    Consider the following morphisms
    $$
    \alpha:A\xrightarrow{\eta_L}\Gamma \xrightarrow{\id_\Gamma\otimes f} \Gamma\otimes_{\eta_R,A,f}B,
    $$
    $$
    \beta:B\otimes_{f,A,\eta_L}\Gamma\otimes_{\eta_R,A,f}B\xrightarrow{\eta_L\otimes g\otimes\eta_R}\Sigma;
    $$
    in particular, there is the natural factorization
    $$
    f:(A,\Gamma)\rightarrow (B,B\otimes_{f,A,\eta_L}\Gamma\otimes_{\eta_R,A,f}B)\rightarrow (B,\Sigma).
    $$
\end{rem}

\begin{thm}[{\cite[Theorem 13]{Naumann07}}]\label{thm:equivalence-of-Hopf-algebroids}
    Let $f=(f_o,f_m):(A,\Gamma)\rightarrow (B,\Sigma)$ be a morphism of flat Hopf algebroids (in the sense that the left and right unit maps are flat), with associated morphisms $\alpha$ and $\beta$ determined in Remark \ref{rem:equivalence-of-Hopf-algebroids}. Then the following statements are equivalence:
    \begin{itemize}
        \item[(i)] $f$ induces an isomorphism of the associated stacks.
        
        \item[(ii)] $f^*:\RCoMod_\Gamma\rightarrow \RCoMod_{\Sigma}$ is an equivalence. 

        \item[(iii)] $\alpha$ is faithfully flat and $\beta$ is an isomorphism.
    \end{itemize}
\end{thm}

\begin{proof}
    Refer to \cite[Theorem 13]{Naumann07} or \cite[Theorem 6.2]{HS05}.
\end{proof}

\begin{construction}\label{construction:base-change-of-Hopf-algebroids}
    Let $(A,\Gamma)$ be a Hopf algebroid, and $f:A\rightarrow B$ a ring homomorphism. As it is mentioned in Remark \ref{rem:equivalence-of-Hopf-algebroids}, the pair of rings $(B,B\otimes_{f,A,\eta_L}\Gamma\otimes_{\eta_R,A,f}B)$ can naturally be equipped with the structure of a Hopf algebroid. More precisely, the left and right units $\eta_L,\eta_R:B\rightarrow B\otimes_{f,A,\eta_L}\Gamma\otimes_{\eta_R,A,f}B$ are given as the inclusions of the first and the third factors; the co-multiplication is determined as the following composite
    \begin{align*}
        B\otimes_{f,A,\eta_L}\Gamma\otimes_{\eta_R,A,f}B
        & \xrightarrow{\id\otimes \Delta\otimes \id}B\otimes_{f,A,\eta_L}(\Gamma\otimes_{\eta_R,A,\eta_L}\Gamma)\otimes_{\eta_R,A,f}B \\
        & \rightarrow B\otimes_{f,A,\eta_L}(\Gamma\otimes_{\eta_R,A} B\otimes_{A,\eta_L}\Gamma)\otimes_{\eta_R,A,f}B \\
        & \xrightarrow{\sim}(B\otimes_{f,A,\eta_L}\Gamma\otimes_{\eta_R,A,f}B)\otimes_{B}(B\otimes_{f,A,\eta_L}\Gamma\otimes_{\eta_R,A,f}B);
    \end{align*}
    the co-unit is determined as the following composite
    $$
    B\otimes_{f,A,\eta_L}\Gamma\otimes_{\eta_R,A,f}B\xrightarrow{\id\otimes \epsilon\otimes\id}B\otimes_A B \xrightarrow{\text{multiplication}} B;
    $$
    finally, the conjugation $c:B\otimes_{f,A,\eta_L}\Gamma\otimes_{\eta_R,A,f}B\rightarrow B\otimes_{f,A,\eta_L}\Gamma\otimes_{\eta_R,A,f}B$ is induced by the conjugation of $\Gamma$ together with interchanging the first and the third factors. It is routine to demonstrate that the axioms for Hopf algebroids are satisfied; write $\Gamma_f:=B\otimes_{f,A,\eta_L}\Gamma\otimes_{\eta_R,A,f}B$, and there is the natural morphism of Hopf algebroids
    $$
    (f,\iota):(A,\Gamma) \rightarrow (B,\Gamma_f).
    $$
\end{construction}

\begin{cor}\label{cor:equivalence-of-Hopf-algebroids}
    Let $(A,\Gamma)$ be a flat Hopf algebroid, and let $f:A\rightarrow B$ be a faithfully flat ring homomorphism. Then the morphism $(f,\iota):(A,\Gamma)\rightarrow (B,\Gamma_f=B\otimes_{f,A,\eta_L}\Gamma\otimes_{\eta_R,A,f}B)$ of Construction \ref{construction:base-change-of-Hopf-algebroids} induces an equivalence of categories 
    $$
    (f,\iota)^*:\RCoMod_\Gamma\rightarrow \RCoMod_{\Gamma_f}.
    $$
\end{cor}

\begin{proof}
    According to Theorem \ref{thm:equivalence-of-Hopf-algebroids}, it suffices to check that, the morphism 
    $$
    \alpha:A\xrightarrow{\eta_L}\Gamma\xrightarrow{\id_\Gamma\otimes f}\Gamma\otimes_{\eta_R,A,f}B
    $$
    is faithfully flat. 

    Note that, since $(A,\Gamma)$ is a flat Hopf algebroid, the left (or right) unit $\eta_L:A\rightarrow \Gamma$ is in fact faithfully flat: $\eta_L$ is flat by assumption, and the axioms of Hopf algebroids implies that, for any $A$-module $M$, the following composite
    $$
    M\xrightarrow{m\mapsto m\otimes 1}M\otimes_{A,\eta_L}\Gamma\xrightarrow{\id\otimes\epsilon}M
    $$
    is the identity of $M$, and therefore $M=0$ if and only if $M\otimes_{A,\eta_L}\Gamma=0$.
    
    It is also clear that, since $f$ is faithfully flat,  $\id_\Gamma\otimes f:\Gamma\rightarrow \Gamma\otimes_{\eta_R,A,f}B$ is also faithfully flat. Therefore $ \alpha:A\xrightarrow{\eta_L}\Gamma\xrightarrow{\id_\Gamma\otimes f}\Gamma\otimes_{\eta_R,A,f}B$ is faithfully flat.
\end{proof}

\begin{construction}\label{construction:functoriality-for-dual-algebras}
    Let $(A,\Gamma)$, $(B,\Sigma)$ be free Hopf algebroids (Convention \ref{convention:the-dual-Hopf-algebroids}), and $f=(f_o,f_m):(A,\Gamma)\rightarrow (B,\Sigma)$ be a morphism of Hopf algebroids such that the morphism $ \alpha:A\xrightarrow{\eta_L}\Gamma\xrightarrow{\id_\Gamma\otimes f}\Gamma\otimes_{\eta_R,A,f}B$ is faithfully flat and makes  $\Gamma\otimes_{\eta_R,A,f}B$ a free $A$-module. Let $\Gamma_f=B\otimes_{f,A,\eta_L}\Gamma\otimes_{\eta_R,A,f}B$, then $(B,\Gamma_f)$ is also a free Hopf algebroid, and $f$ admits the following factorization 
    $$
    (A,\Gamma)\xrightarrow{(f,\iota)} (B,\Gamma_f)\xrightarrow{(\id,\beta)} (B,\Sigma),
    $$
    and the morphism $(\id,\beta):(B,\Gamma_f)\rightarrow (B,\Sigma)$ immediately induces a continuous homomorphism of topological algebras 
    $$
    \beta^\vee:\Sigma^\vee\rightarrow \Gamma_f^\vee,
    $$
    which further induces, through restriction, a functor 
    $$
    \beta^\vee_*:\Mod^d_{\Gamma_f^\vee}\rightarrow \Mod^d_{\Sigma^\vee}. 
    $$
    Recall that, by Proposition \ref{thm:covariant-translation-for-profinitely-generated}, there are the equivalences of categories $\RCoMod_{\Gamma_f}\simeq \Mod^d_{\Gamma_f^\vee}$, $\RCoMod_{\Sigma}\simeq \Mod^d_{\Sigma^\vee}$, and, combining Theorem \ref{thm:equivalence-of-Hopf-algebroids}, there are the equivalences of categories $\RCoMod_{\Gamma}\simeq \RCoMod_{\Gamma_f}\simeq \Mod^d_{\Gamma_f^\vee}$.
\end{construction}

\begin{prop}\label{prop:functoriality-for-dual-algebras}
    In Construction \ref{construction:functoriality-for-dual-algebras}, under the equivalences of categories $\RCoMod_{\Gamma_f}\simeq \Mod^d_{\Gamma_f^\vee}$, and $\RCoMod_{\Gamma}\simeq \RCoMod_{\Gamma_f}\simeq \Mod^d_{\Gamma_f^\vee}$, the functor $f^*:\RCoMod_\Gamma\rightarrow \RCoMod_{\Sigma}$ can be identified as $\beta^\vee_*:\Mod^d_{\Gamma_f^\vee}\rightarrow \Mod^d_{\Sigma^\vee}$.
\end{prop}

\begin{proof}
    It suffices to check that the functor $\beta^\vee_*$, up to the stated equivalences of categories, is exactly $(\id,\beta)^*$, and this is clear.
\end{proof}

\begin{rem}\label{rem:functoriality-for-comodules-right-adjoint}
    Let $f=(f_o,f_m):(A,\Gamma)\rightarrow (B,\Sigma)$ be a morphism of Hopf algebroids. The functor $f^*:\RCoMod_\Gamma\rightarrow \RCoMod_\Sigma$ (mentioned in Remark \ref{rem:functoriality-for-comodules}) admits a right adjoint $f_*:\RCoMod_\Sigma\rightarrow \RCoMod_\Gamma$, which can be determined as follows.

    Given a $\Sigma$-comodule $N$. Through the functor $f^*$, $\Gamma\otimes_{\eta_R,A}B$ is equipped with the structure of a right $\Sigma$-comodule. Consider the comodule tensor (see Remark \ref{rem:symmetric-monoidal-structure-for-dual}) of $N$ and $\Gamma\otimes_{\eta_R,A}B$:
    $$
    (\Gamma\otimes_{\eta_R,A}B)\otimes_B N\simeq \Gamma\otimes_{\eta_R,A}N,
    $$
    then the group of $\Sigma$-comodule homomorphisms $\Hom_\Sigma(B,\Gamma\otimes_{\eta_R,A}N)$ is in fact a sub-left $\Gamma$-comodule of $\Gamma\otimes_{\eta_R,A}N$, where the structure map for the latter is $\Delta\otimes\id_N$. Finally, by applying the conjugation, this left $\Gamma$-comodule structure of $\Hom_\Sigma(B,\Gamma\otimes_{\eta_R,A}N)$ becomes a right $\Gamma$-comodule structure. Therefore the assignment $N\mapsto \Hom_\Sigma(B,\Gamma\otimes_{\eta_R,A}N)$ determines a functor $\RCoMod_\Sigma\rightarrow \RCoMod_\Gamma$, which turns out to be essentially $f_*$. 
\end{rem}

\begin{rem}\label{rem:functoriality-for-dual-algebras-right-adjoint}
    Resume Construction \ref{construction:functoriality-for-dual-algebras}. The homomorphism of topological algebras $\beta^\vee:\Sigma^\vee\rightarrow \Gamma_f^\vee$ also induces a functor 
    $$
    \Hom_{\Sigma^\vee}(\Gamma_f^\vee,-):\Mod_{\Sigma^\vee}\rightarrow \Mod_{\Gamma^\vee_f,}
    $$
    which restricts to a functor 
    $$
    \beta^{\vee !}:\Mod^d_{\Sigma^\vee}\rightarrow \Mod^d_{\Gamma^\vee_f}.
    $$
    Indeed, for any discrete $\Gamma^\vee$-module $N$, by Remark \ref{rem:tensor-Hom-profinitely-generated-topology}, the compact-open topology for the group of continuous homomorphisms $\Hom_{\Sigma^\vee}(\Gamma^\vee_f,N)$ is discrete. 

    Under the equivalences of categories $\RCoMod_{\Gamma_f}\simeq \Mod^d_{\Gamma_f^\vee}$, and $\RCoMod_{\Gamma}\simeq \RCoMod_{\Gamma_f}\simeq \Mod^d_{\Gamma_f^\vee}$, the functor $f_*:\RCoMod_\Sigma\rightarrow \RCoMod_\Gamma$ should be naturally isomorphic to $\beta^{\vee !}:\Mod^d_{\Sigma^\vee}\rightarrow \Mod^d_{\Gamma_f^\vee}$.
\end{rem}

\begin{rem}\label{rem:functoriality-for-dual-algebras-left-left-adjoint}
    Resume Construction \ref{construction:functoriality-for-dual-algebras}. The homomorphism of topological algebras $\beta^\vee:\Sigma^\vee\rightarrow \Gamma^\vee_f$ induces by base change a functor 
    $$
    \Gamma^\vee_f\otimes_{\Sigma^\vee}-:\Mod_{\Sigma^\vee}\rightarrow \Mod_{\Gamma^\vee_f},
    $$
    which, however, does not immediately give rise to a functor 
    $$
    \beta^{\vee *}:\Mod^d_{\Sigma^\vee}\rightarrow \Mod^d_{\Gamma_f^\vee}.
    $$
\end{rem}

\section{The Symmetric Monoidal Structure}\label{section:the-symmetric-monoidal-structure}

In this section we review the symmetric monoidal structure for comodules over Hopf algebroids from the perspective of corresponding dual algebras introduced in Section \ref{section:the-dual-algebras}. Keep Convention \ref{convention:the-dual-Hopf-algebroids}.

\begin{rem}\label{rem:symmetric-monoidal-structure-for-dual}
    The equivalence of categories $\RCoMod_\Gamma\xrightarrow{\sim}\Mod^d_{\Gamma^\vee}$, established in Proposition \ref{thm:covariant-translation-for-profinitely-generated}, is clearly an equivalence of abelian categories. Moreover, the category $\RCoMod_\Gamma$ admits a symmetric monoidal structure induced by the $A$-algebra structure of $\Gamma$: let $M,N$ be right $\Gamma$-comodules with comodule structure maps $\psi_M,\psi_N$, then $M\otimes_AN$ can be equipped with the right $\Gamma$-comodule structure map 
    \begin{align*}
        \psi_{M\otimes_A N}:M\otimes_A N & \xrightarrow{\psi_M\otimes\psi_N} (M\otimes_{A,\eta_L}\Gamma)\otimes_{\eta_R,A,\eta_R}(\Gamma\otimes_{\eta_L,A}N) \\
        & \xrightarrow{\sim} M\otimes_{A,\eta_L}(\Gamma\otimes_{\eta_R,A,\eta_R}\Gamma)\otimes_{\eta_L,A}N \\
        & \xrightarrow{\id_M\otimes \text{multiplication}\otimes \id_N}
        M\otimes_{A,\eta_L}\Gamma\otimes_{\eta_L,A}N \\
        &\xrightarrow{\sim}(M\otimes_A N)\otimes_{A,\eta_L}\Gamma,
    \end{align*}
    note that there is the commutative diagram 
    $$
    \xymatrix@R=50pt@C=75pt{
    A\otimes_\Z A \ar[r]^-{\text{multiplication}} \ar[d]^{\eta_L\otimes \eta_L} & A \ar[d]^{\eta_L} \\
    \Gamma\otimes_{\eta_R,A,\eta_R}\Gamma \ar[r]^-{\text{multiplication}} & \Gamma.
    }
    $$
    
    Through the equivalence $\RCoMod_\Gamma\xrightarrow{\sim}\Mod^d_{\Gamma^\vee}$, the category $\Mod^d_{\Gamma^\vee}$ certainly admits a symmetric monoidal structure, which should correspond to a ``co-multiplication'' of $\Gamma^\vee$.
\end{rem}

\begin{rem}\label{rem:co-multiplication-for-profinitely-generated}
    The Hopf algebroid structure of $(A,\Gamma)$ implies that, the multiplication $\Gamma\otimes_\Z\Gamma\rightarrow \Gamma$ can be linearized as $\Gamma\otimes_{A-A}\Gamma\rightarrow \Gamma$, where $\Gamma$ is regarded as an $A$-$A$-bimodule through the left and right units $\eta_L,\eta_R:A\rightarrow \Gamma$. By taking the $\eta_L$-linear dual, there is a morphism 
    $$
    \Gamma^\vee\rightarrow \Hom_{A,\eta_L}(\Gamma\otimes_{A-A}\Gamma,A),
    $$
    which, according to Remark \ref{rem:symmetric-monoidal-structure-for-dual}, is involved in the symmetric monoidal structure of $\RCoMod_\Gamma\simeq \Mod^d_{\Gamma^\vee}$. In fact, the Hopf algebroid structure of $(A,\Gamma)$ can naturally induce a Hopf algebroid structure on the pair of rings $(A,\Gamma\otimes_{A-A}\Gamma)$. 


\end{rem}

\begin{construction}\label{construction:tensor-Hopf-algebroid}
    For the pair of rings $(A,\Gamma\otimes_{A-A}\Gamma)$, the following morphisms are natural:
    \begin{itemize}
        \item the left unit map $\eta_L:A\xrightarrow{\text{left factor}}A\otimes A\rightarrow \Gamma\otimes_{A-A}\Gamma$, and the right unit map $\eta_R:A\xrightarrow{\text{right factor}}A\otimes A\rightarrow \Gamma\otimes_{A-A}\Gamma$;

        \item the co-unit map $\epsilon:\Gamma\otimes_{A-A}\Gamma\xrightarrow{\epsilon_\Gamma\otimes\epsilon_\Gamma}A\otimes A\xrightarrow{\text{multiplication}}A$, where $\epsilon_\Gamma:\Gamma\rightarrow A$ is the co-unit map for $(A,\Gamma)$;

        \item the conjugation map $c:\Gamma\otimes_{A-A}\Gamma\xrightarrow{c_\Gamma\otimes c_\Gamma}\Gamma\otimes_{A-A}\Gamma$, where $c_\Gamma:\Gamma\rightarrow \Gamma$ is the conjugation for $(A,\Gamma)$, and note that $c_\Gamma\otimes c_\Gamma$ also interchanges the two factors;

        \item the co-multiplication map is determined as follows. First consider the following morphism 
        $$
        \Gamma\otimes_{A-A}\Gamma\xrightarrow{\Delta_\Gamma\otimes \Delta_\Gamma}(\Gamma\otimes_{\eta_R,A,\eta_L}\Gamma)\otimes_{A-A}(\Gamma\otimes_{\eta_R,A,\eta_L}\Gamma),
        $$
        where $\Delta_\Gamma:\Gamma\rightarrow \Gamma\otimes_{\eta_R,A,\eta_L}\Gamma$ is the co-multiplication for $(A,\Gamma)$, and $\Gamma\otimes_{\eta_R,A,\eta_L}\Gamma$ is regarded as an $A$-$A$-bimodule through the left most and right most $A$-module structures; then, pair the first factor with the third, and the second with the forth, we have the following composite as the co-multiplication
    \begin{align*}
        \Delta:\Gamma\otimes_{A-A} 
        \Gamma
        & \xrightarrow{\Delta_\Gamma\otimes \Delta_\Gamma}(\Gamma\otimes_{\eta_R,A,\eta_L}\Gamma)\otimes_{A-A}(\Gamma\otimes_{\eta_R,A,\eta_L}\Gamma)\\
        \ & \rightarrow (\Gamma\otimes_{A-A}\Gamma)\otimes_{\eta_R,A,\eta_L}(\Gamma\otimes_{A-A}\Gamma).
    \end{align*}
    \end{itemize}
    The multiplication of $\Gamma$ then induces a morphism of Hopf algebroids $(A,\Gamma\otimes_{A-A}\Gamma)\rightarrow (A,\Gamma)$, which further induces a morphism of the corresponding dual algebras $\Gamma^\vee\rightarrow (\Gamma\otimes_{A-A}\Gamma)^\vee$.
\end{construction}

\begin{rem}\label{rem:symmetric-monoidal-structure-for-dual-3}
    In practice usually we would work with a chosen basis of $\Gamma$ as a free $A,\eta_L$-module. However, for the symmetric monoidal structure of $\RCoMod_\Gamma$ as described in Remark \ref{rem:symmetric-monoidal-structure-for-dual}, the consideration of the $\eta_R$-structure would prevent the usage of the chosen basis of $\Gamma$ as a free $A,\eta_L$-module. 
    
    For the same reason, suppose that $M,N\in \RCoMod_\Gamma\simeq \Mod^d_{\Gamma^\vee}$, and given $D_1,D_2\in \Gamma^\vee$, there is no $D_1\otimes D_2:M\otimes_A N\rightarrow M\otimes_A N$; yet there is nevertheless $D_1\otimes D_2:M\otimes_\Z N\rightarrow M\otimes_A N$ as the following composite  
    \begin{align*}
        M\otimes_\Z N
        & \xrightarrow{\psi_M\otimes\psi_N}(M\otimes_{A,\eta_L}\Gamma)\otimes_\Z (\Gamma\otimes_{\eta_L,A}N) \\
        & \xrightarrow{(\id\otimes D_1)\otimes(D_2\otimes\id)}M\otimes_\Z N\rightarrow M\otimes_A N.
    \end{align*}
    Now, suppose that a basis of $\Gamma$ as a free $A,\eta_L$-module is specified, say, $\{e_i\}_{i\in I}$, and let $\{e_i^\vee\}_{i\in I}$ denote the corresponding dual basis in $\Gamma^\vee$; then $\Hom_{A,\eta_L}(\Gamma\otimes_{\eta_L,A,\eta_L}\Gamma,A)$ admits the (topological) basis $\{e_i^\vee\otimes e_j^\vee\}_{i,j\in I}$. For an element $D\in \Gamma^\vee$, let $d_{ij}$ denote the image of $e_i\otimes e_j$ through the composite 
    $$
    \Gamma\otimes_{\eta_L,A,\eta_L}\Gamma \xrightarrow{\text{multiplication}}\Gamma\xrightarrow{D}A;
    $$
    then for $m\in M,n\in N$, 
    $$
    D(m\otimes n)=\sum_{i,j\in I}d_{i,j}(e_i^\vee(m)\otimes e_j^\vee(n))\in M\otimes_A N, 
    $$
    where the summation at the right hand side is essentially finite, since for all but finitely many $i,j$, we have $e_i^\vee(m)=e_j^\vee(n)=0$; finally note that, although each $e_i^\vee\otimes e_j^\vee$ only gives a morphism $M\otimes_\Z N\rightarrow M\otimes_A N$, the above formula actually determines a morphism $D:M\otimes_A N\rightarrow M\otimes_A N$.
\end{rem}

\begin{rem}\label{rem:left-and-right-adjoints-to-multiplication}
    The multiplication map $A\otimes_\Z A\rightarrow A$ induces the ``forgetful'' functor $\Mod_A\rightarrow \Mod_{A\otimes_\Z A}$. Its left adjoint is $-\otimes_{A\otimes_\Z A}A$, and its right adjoint is $\Hom_{A\otimes_\Z A}(A,-)$. Also note that, for an $A$-$A$-bimodule, or an $A\otimes_\Z A$-module $M$, $\Hom_{A\otimes_\Z A}(A,M)$ is precisely the $A\otimes_\Z A$-submodule of $M$ consists of elements $m\in M$ such that $(a\otimes 1)m=(1\otimes a)m$. 

    Let $M,M'$ be $A$-modules, $N$ be an abelian group. Then the canonical inclusion 
    $$
    \Hom_\Z(M\otimes_A M',N)\hookrightarrow \Hom_\Z(M\otimes_\Z M', N)
    $$
    induces an isomorphism 
    $$
    \Hom_\Z(M\otimes_A M',N)\xrightarrow{\sim} \Hom_{A\otimes_\Z A}(A,\Hom_\Z(M\otimes_\Z M',N)),
    $$
    where $\Hom_\Z(M\otimes_\Z M',N)$ is regarded as an $A\otimes A$-module through the $A\otimes_\Z A$-module structure of $M\otimes_\Z M'$. Indeed, it suffices to apply the adjunction that 
    \begin{align*}
        \Hom_{A\otimes_\Z A}(A,\Hom_\Z(M\otimes_\Z M',N))
        & \simeq \Hom_\Z((M\otimes_\Z M')\otimes_{A\otimes_\Z A}A,N) \\
        & \simeq \Hom_\Z(M\otimes_A M',N).
    \end{align*}
\end{rem}

\begin{rem}\label{rem:profinite-type-tensor-product}
    If there is a good notion of topological tensor product for topological abelian groups like $\Gamma^\vee$, then the symmetric monoidal structure on $\RCoMod_\Gamma\simeq \Mod^d_{\Gamma^\vee}$ could be reflected to a co-multiplication of the form $\Gamma^\vee\rightarrow \Gamma^\vee\widehat{\otimes}_{left,A,left}\Gamma^\vee$, or more accurately, the morphism $\Gamma^\vee\rightarrow \Hom_{A,\eta_L}(\Gamma\otimes_{A-A}\Gamma,A)$ in Construction \ref{rem:co-multiplication-for-profinitely-generated} could be interpreted as a morphism 
    $$
    \Gamma^\vee\rightarrow \Hom_{A\otimes_\Z A}(A,\Gamma^\vee\widehat{\otimes}_{left,A,left}\Gamma^\vee).
    $$
    A candidate for this notion of tensor product is that of the solid tensor of solid modules.
\end{rem}

\section{Profinite Rings and Dual Algebras}\label{section:a-review-of-profinite-rings}

In order to illustrate the possible study of comodule cohomology through the perspective of dual algebras, in this section we review the classical results of profinite rings, especially the corresponding cohomological algebra, for which the main reference is \cite{RZ10}. A comparison of the classical approach to cohomological algebra with the cohomological algebra in solid modules can be found in \cite{Tang24}.

\begin{defn}[{\cite[Proposition 5.1.2]{RZ10}}]\label{defn:profinite-rings}
    A profinite ring is a topological ring $\Lambda$ which satisfies the following equivalent conditions:
    \begin{itemize}
        \item[(1)] $\Lambda$ is isomorphic to a cofiltered limit of finite discrete rings. 

        \item[(2)] $\Lambda$ is compact and Hausdorff.

        \item[(3)] $\Lambda$ is compact, Hausdorff, and totally disconnected. 

        \item[(4)] The zero element of $\Lambda$ has a fundamental neighborhood $\{T_i\}_{i\in I}$ of open ideals, and $A\xrightarrow{\sim}\varprojlim_{i\in I}A/T_i$.

        \item[(5)] There is a cofiltered diagram $\{\Lambda_i,\varphi_{ij}\}$ where each $\varphi_{ij}$ is an epimorphism, so that there is an isomorphism $\Lambda\simeq \varprojlim_{i\in I}\Lambda_i$.    
    \end{itemize}
    In particular, every finite discrete ring is a profinite ring. 
\end{defn}

\begin{rem}\label{rem:topological-modules}
    Let $\Lambda$ be a profinite ring. By a topological $\Lambda$-module it means a Hausdorff topological abelian group $M$ equipped with a continuous $\Lambda$-action $\Lambda\times M\rightarrow M$ which is bi-additive and associate. 

    Let $M,N$ be two topological $\Lambda$-modules. The group of continuous $\Lambda$-homomorphisms $\Hom_{\Lambda}(M,N)$ can be equipped with the compact-open topology, and is therefore a topological abelian group. 
\end{rem}

Let $\Lambda$ be a profinite ring, and let $\Hom_\Lambda(-,-)$ denote the group of continuous $\Lambda$-homomorphisms between topological $\Lambda$-modules. In order to develop a theory of cohomological algebra, we specify the following two types of topological $\Lambda$-modules.

\begin{defn}\label{defn:discrete-modules}
    A discrete $\Lambda$-module is a topological $\Lambda$-module $N$ whose underlying topological space is discrete, that is, $N$ has the following property (\cite[Lemma 5.1.1]{RZ10}):
    \begin{itemize}
        \item $N$ is the union of its finite $\Lambda$-submodules.
    \end{itemize}
    Let $\Mod^d_\Lambda$ denote the category of discrete $\Lambda$-modules. It is routine to check that $\Mod^d_\Lambda$ is an abelian category.
\end{defn}

\begin{defn}\label{defn:profinite-modules}
    A profinite $\Lambda$-module is a topological $\Lambda$-module $M$ whose underlying topological abelian group is profinite, that is, $M$ has the following property (\cite[Lemma 5.1.1]{RZ10}):
    \begin{itemize}
        \item $M$ is the inverse limit of its finite quotient $\Lambda$-modules; equivalently, the submodules of $M$ of finite index form a fundamental system of neighborhoods of $0$.
    \end{itemize}
    Let $\Mod^{pf}_\Lambda$ denote the category of profinite $\Lambda$-modules. It is routine to check that $\Mod^{pf}_\Lambda$ is an abelian category. 
\end{defn}

\begin{rem}\label{rem:profinite-tensor}
    Let $\Mod^{pf}_{\Lambda^{\op}}$ denote the category of right profinite $\Lambda$-modules. There is the profinite tensor product
    $$
    \widehat{\otimes}_\Lambda:\Mod^{pf}_{\Lambda^{\op}}\times \Mod^{pf}_{\Lambda}\rightarrow \PfAb,
    $$
    where $\PfAb$ is the category of profinite abelian groups. This profinite tensor product can be characterized by the standard universal property, or more concretely, given a right profinite $\Lambda$-module $M$ which is expressed as a cofiltered limit of finite right $\Lambda$-modules $M\simeq \varprojlim_{i\in I}M_i$, and a left profinite $\Lambda$-module $N$ which is expressed as a cofiltered limit of finite left $\Lambda$-modules $N\simeq \varprojlim_{j\in J}N_j$, then 
    $$
    M\widehat{\otimes}_\Lambda N\simeq \varprojlim_{i\in I,j\in J}M_i\otimes_{\Lambda}N_j.
    $$

\end{rem}


\begin{rem}\label{rem:profinite-to-discrete}
    Let $M$ be a profinite $\Lambda$-module, which is written as a cofiltered limit of its finite quotient $\Lambda$-modules: $M\simeq \varprojlim_{i\in I}M_i$; and let $N$ be a discrete $\Lambda$-module. Then the canonical morphism 
    $$
    \varinjlim_{i\in I^{\op}}\Hom_{\Lambda}(M_i,N)\rightarrow \Hom_\Lambda(\varprojlim_{i\in I}M_i,N)
    $$
    is an isomorphism (\cite[Lemma 5.1.4]{RZ10}). In particular, the compact-open topology for $\Hom_\Lambda(M,N)$ is discrete. 
\end{rem}


\begin{prop}[Pontryagin Duality]\label{prop:Pontryjagin-dual}
    The assignment $M\mapsto \Hom_\Z(M,\Q/\Z)$ induces an equivalence of categories $(\Mod^{pf}_\Lambda)^\op\xrightarrow{\sim}\Mod^d_{\Lambda^{\op}}$. 
\end{prop}

\begin{proof}
    This is an application of the Pontryagin duality between discrete torsion groups and profinite groups, see \cite[Theorem 2.9.6]{RZ10} and \cite[$\S$ 5.1]{RZ10}.
\end{proof}

\begin{rem}[{\cite[Propositions 5.4.2, 5.4.4]{RZ10}}]\label{rem:projective-and-injective}
    The category $\Mod^{pf}_\Lambda$ has enough projectives, and a profinite $\Lambda$-module is projective if and only if it is a direct summand of a free profinite $\Lambda$-module.

    Dually, the category $\Mod^d_\Lambda$ has enough injectives, and a discrete $\Lambda$-module is injective if and only if it is a direct factor of a cofree discrete $\Lambda$-module.
\end{rem}

\begin{lem}[{\cite[Exercise 5.4.7]{RZ10}}]\label{lem:projective-and-injective}
    Let $P$ be a projective profinite $\Lambda$-module, then the functor 
    $$
    \Hom_{\Lambda}(P,-):\Mod^d_\Lambda\rightarrow \Ab
    $$
    is exact. 
    
    Similarly, let $I$ be an injective discrete $\Lambda$-module, then the functor 
    $$
    \Hom_\Lambda(-,I):(\Mod^{pf}_\Lambda)^\op\rightarrow \Ab
    $$
    is exact.
\end{lem}

\begin{proof}
    Since the two statements are dual to each other by Pontryagin duality, 
    it suffices to check the first statement, that is, for any epimorphism $r:N\rightarrow N'$ of discrete $\Lambda$-modules, the corresponding morphism $\Hom_\Lambda(P,N)\xrightarrow{r\circ -} \Hom_\Lambda(P,N')$ should be an epimorphism. Given any continuous $\Lambda$-homomorphism $f:P\rightarrow N'$, since $P$ is compact and $N'$ discrete, $f(P)$ is a finite $\Lambda$-submodule of $N'$. Note that, $r^{-1}(f(P))$, being a discrete $\Lambda$-submodule of $N$, is a union of its finite $\Lambda$-submodules, and therefore there is a finite $\Lambda$-submodule $\widetilde{N}\subseteq r^{-1}(f(P))\subseteq N$ such that $r|_{\widetilde{N}}:\widetilde{N}\rightarrow f(P)$ is an epimorphism. Therefore, it reduces to the case where $N,N'$ are finite $\Lambda$-modules, which are certainly profinite $\Lambda$-modules, and the statement follows. 
\end{proof}

\begin{rem}\label{rem:cohomological-algebra-for-profinite}
    Although currently we have not yet set up a suitable category of $\Lambda$-modules encompassing both the profinite and discrete $\Lambda$-modules for cohomological algebra, Remark \ref{rem:projective-and-injective} and Lemma \ref{lem:projective-and-injective} guarantees that $\Ext_{\Lambda}(M,N)$ is well-defined in the following three cases:
    \begin{itemize}
        \item Both $M$ and $N$ are profinite $\Lambda$-modules, and $\Ext_{\Lambda}(M,N)$ can be computed by taking a resolution of $M$ by projective profinite $\Lambda$-modules. Note that in this case, $\Hom_{\Lambda}(M,N)$ might admit a non-discrete topology, but in $\Ext_\Lambda(M,N)$ we do not take this topology into account.

        \item $M$ is a profinite $\Lambda$-module, $N$ is a discrete $\Lambda$-module, and $\Ext_{\Lambda}(M,N)$ can be computed by taking a resolution of $M$ by projective profinite $\Lambda$-modules, and can also be computed by taking a resolution of $N$ by injective discrete $\Lambda$-modules. 

        \item Both $M$ and $N$ are discrete $\Lambda$-modules, and $\Ext_{\Lambda}(M,N)$ can be computed by taking a resolution of $N$ by injective discrete $\Lambda$-modules. Note that in this case, $\Hom_{\Lambda}(M,N)$ might admit a non-discrete topology, but in $\Ext_\Lambda(M,N)$ we do not take this topology into account.
    \end{itemize}
\end{rem}

Now we apply the aforementioned setting to those Hopf algebroids whose dual algebras are profinite rings.

\begin{convention}\label{convention:Hopf-algebroids-profinite-dual}
    In the rest of this section, let $(A,\Gamma)$ be a Hopf algebroid where $A$ is a finite commutative ring and $\Gamma$, through the left unit map $\eta_L:A\rightarrow \Gamma$, is a free $A$-module with countable basis.
\end{convention}

\begin{lem}\label{lem:profinite-dual-algebra}
    Let $(A,\Gamma)$ be a Hopf algebroid as in Convention \ref{convention:Hopf-algebroids-profinite-dual}. Then its dual algebra $\Gamma^\vee$ is a profinite ring.
\end{lem}

\begin{proof}
    It suffices to combine Lemma \ref{lem:dual-of-free-elaborated} and Definition \ref{defn:profinite-rings}.
\end{proof}

Recall from Proposition \ref{thm:covariant-translation-for-profinitely-generated} that there is an equivalence of abelian categories $\RCoMod_\Gamma\simeq \Mod^d_{\Gamma^\vee}$, where $\RCoMod_\Gamma$ is the category of right $\Gamma$-comodules.


\begin{defn}\label{defn:cohomology-for-comodules-and-modules}
    Suppose that $N$ is a $\Gamma$-comodule or a discrete $\Gamma^\vee$-module. Let $\RHom_{\Gamma}(N,-)\simeq \RHom_{\Gamma^\vee}(N,-)$ denote the derived functor of 
    $$
    \Hom_{\Gamma}(N,-)\simeq \Hom_{\Gamma^\vee}(N,-):\RCoMod_\Gamma\simeq \Mod^d_{\Gamma^\vee}\rightarrow \Ab,
    $$ and set $\Ext_{\Gamma^\vee}(N,N'):=\RHom_{\Gamma^\vee}(N,-)(N')$. Also, let $\RHom_{\Gamma^\vee}(-,N)$ denote the derived functor of 
    $$
    \Hom_{\Gamma^\vee}(-,N):(\Mod_{\Gamma^\vee}^{pf})^\op\rightarrow \Ab,
    $$
    and set $\Ext_{\Gamma^\vee}(M,N):=\RHom_{\Gamma^\vee}(-,N)(M)$.

    Similarly, suppose that $M$ is a profinite $\Gamma^\vee$-module, let $\RHom_{\Gamma^\vee}(-,M)$ denote the derived functor of 
    $$
    \Hom_{\Gamma^\vee}(-,M):(\Mod^{pf}_{\Gamma^\vee})^\op\rightarrow \Ab,
    $$ 
    and set $\Ext_{\Gamma^\vee}(M',M):=\RHom_{\Gamma^\vee}(-,M)(M')$. Also, let $\RHom_{\Gamma^\vee}(M,-)$ denote the derived functor of 
    $$
    \Hom_{\Gamma^\vee}(M,-):\Mod^d_{\Gamma^\vee}\rightarrow \Ab,
    $$
    and set $\Ext_{\Gamma^\vee}(M,N):=\RHom_{\Gamma^\vee}(M,-)(N)$.

    Note that if $M$ is a profinite $\Gamma^\vee$-module and $N$ is a discrete $\Gamma^\vee$-module, the notation $\Ext_{\Gamma^\vee}(M,N)$ is potentially ambiguous; however, by Remark \ref{rem:cohomological-algebra-for-profinite}, $\RHom_{\Gamma^\vee}(-,N)(M)\simeq \RHom_{\Gamma^\vee}(M,-)(N)$ and there is no essential ambiguity. 
\end{defn}

\begin{prop}\label{prop:Pontryagin-duality-for-profinite-dual}
    The Pontryagin duality (Proposition \ref{prop:Pontryjagin-dual}) induces an equivalence of categories $(\Mod^{pf}_{\Gamma^\vee})^\op\xrightarrow{\sim}\Mod^d_{\Gamma^{\vee}_\op}$. 
\end{prop}

\begin{proof}
    This is immediate by Proposition \ref{prop:Pontryjagin-dual}.
\end{proof}




\section{Further Development}\label{section:further-development}

In this final section we include some ideas which has not yet been developed in this article but are necessarily demanded by both theory and applications. Still, let $(A,\Gamma)$ be a Hopf algebroid as in Convention \ref{convention:the-dual-Hopf-algebroids}.

\begin{rem}\label{rem:graded-and-filtered}
    The Hopf algebroids in consideration might be graded or filtered. In that case, there are potentially two approaches to dual algebras: either we consider the graded or filtered $\Hom$ and obtain dual algebras as (topological) graded or filtered algebras; or, recall that, graded abelian groups are comodules over the Hopf algebroid $(\Z,\Z[t,t^{-1}])$, and filtered abelian groups are comodules over the Hopf algebroid $(\Z[\tau],\Z[\tau][t,t^{-1}])$ (see for example \cite{Moulinos21}), thus a graded Hopf algebroid might be treated as an ordinary Hopf algebroid of the form $(A,\Gamma[t,t^{-1}])$, and a filtered Hopf algebroid might be treated as an Hopf algebroid $(A,\Gamma[t,t^{-1}])$ which is further equipped with an endomorphism $\tau$.  
\end{rem}

\begin{rem}\label{rem:the-dualizing-object}
  As it constantly reappears throughout this article, a complete treatment of the dual algebra $\Gamma^\vee$ should be carried out in the context of solid abelian groups. 
  In particular, we should have the dualizing object $\omega:=\RHom_{\Gamma^\vee}(A,\Gamma^\vee)$ as an object in the derived category of right $\Gamma^\vee$-modules. 
\end{rem}

\begin{ex}\label{ex:the-prismatic-dualizing-object}
    Resume Example \ref{ex:the-prismatic-dual}, and also recall from Example \ref{ex:prismatic-cohomology-through-dual} that, there should be following resolution of left $\Gamma^\vee$-modules:
    $$
    0\rightarrow \Gamma^\vee\xrightarrow{-\cdot\nabla_q}\Gamma^\vee\xrightarrow{\eta_L}A\rightarrow 0,
    $$
    suggesting that $\RHom_{\Gamma^\vee}(A,\Gamma^\vee)$ is computed by the complex 
    $$
    \Gamma^\vee\xrightarrow{\nabla_q\cdot-}\Gamma^\vee,
    $$
    which is quasi-isomorphic to $(\Gamma^\vee/\nabla_q\Gamma^\vee)[-1]$. Note that through $A\xrightarrow{\epsilon^\vee}\Gamma^\vee\rightarrow \Gamma^\vee/\nabla_q\Gamma^\vee$, the underlying $A$-module of $\Gamma^\vee/\nabla_q\Gamma^\vee$ is isomorphic to $A$; but under this identification, for $\nabla_q\in \Gamma^\vee$ and $f\in A=\Z_p[[q-1]]$, we have 
    $$
    f\cdot \nabla_q=-\frac{f(q)-f(q^{\frac{1}{1+p}})}{q(q-1)}.
    $$
    Note also that, regarding $A$ as a right $\Gamma^\vee$-module in this way, then by definition there should be following resolution of right $\Gamma^\vee$-modules:
    $$
    0\rightarrow \Gamma^\vee\xrightarrow{\nabla_q\cdot -}\Gamma^\vee\rightarrow A\rightarrow 0,
    $$
    suggesting that $\RHom_{\Gamma^\vee_\op}(A,\Gamma^\vee_\op)$ is computed by the complex
    $$
    \Gamma^\vee\xrightarrow{-\cdot \nabla_q}\Gamma^\vee,
    $$
    which is quasi-isomorphic to $A[-1]$, where $A$ is naturally equipped with its canonical left $\Gamma^\vee$-module structure.
\end{ex}


\begin{rem}\label{rem:the-dualizing-object-comodule-version}
    There might be another possible computation of the dualizing object for $(A,\Gamma)$: regard $A$ as a left $\Gamma$-comodule through $\eta_L:A\rightarrow \Gamma\simeq \Gamma\otimes_{\eta_R,A}A$, and $\Gamma$ as a left $\Gamma$-comodule through the co-multiplication $\Gamma\xrightarrow{\Delta}\Gamma\otimes_{\eta_R,A,\eta_L}\Gamma$; then there is the object $\RHom_{\LCoMod_\Gamma}(\Gamma,A)$, which is naturally a right module over the algebra $\Hom_{\LCoMod_\Gamma}(\Gamma,\Gamma)\simeq \Hom_{A,\eta_L}(\Gamma,A)$. Note that, potentially the object $\RHom_{\LCoMod_\Gamma}(\Gamma,A)$ should also be ``equipped with a topology'', although which might turn out to be discrete.
\end{rem}

\begin{ex}\label{ex:the-prismatic-dualizing-object-comodule-version}
    Resume Example \ref{ex:the-prismatic-dual}, that is, consider the Hopf algebroid $(A,\Gamma)=(\Z_p[[q-1]],\Z_p[[q_0-1,q_1-1]]\{\frac{q_1-q_0}{[p]_{q_0}}\}_\delta)$, together with the ring homomorphism 
    $$
    m_q:\Z_p[[q_0-1,q_1-1]]\{\frac{q_1-q_0}{[p]_{q_0}}\}_\delta\rightarrow \Z_p[[q-1]],\quad\quad q_0\mapsto q,q_1\mapsto q^{1+p},
    $$
    and the operator $\nabla_q=\frac{1}{q(q-1)}\circ(m_q-1)\in \Gamma^\vee=\Hom_{A,\eta_L}(\Gamma,A)$. If $\Gamma$ is regarded as a right $\Gamma$-comodule through $\Delta:\Gamma\rightarrow \Gamma\otimes_{\eta_R,A,\eta_L}\Gamma$, then $\nabla_q$ acts on $\Gamma$ through the composite (and we abusively denote this composite still by $\nabla_q$)
    $$
    \nabla_q:\Gamma\xrightarrow{\Delta}\Gamma\otimes_{\eta_R,A,\eta_L}\Gamma\xrightarrow{\id_\Gamma\otimes \nabla_q}\Gamma,
    $$
    which is not a morphism of right $\Gamma$-comodule (not even a morphism of $A,\eta_R$-modules), yet is a morphism of left $\Gamma$-comodules; moreover, there is the following ($(p,[p]_q)$-completely) exact sequence 
    $$
    0\rightarrow A\xrightarrow{\eta_L}\Gamma\xrightarrow{\nabla_q}\Gamma\rightarrow 0,
    $$
    suggesting that $\RHom_{\LCoMod_\Gamma}(\Gamma,A)$ is computed by the complex 
    $$
    (\Hom_{\LCoMod_\Gamma}(\Gamma,\Gamma)\xrightarrow{\nabla_q\circ-}\Hom_{\LCoMod_\Gamma}(\Gamma,\Gamma))\simeq (\Gamma^\vee\xrightarrow{\nabla_q\circ -}\Gamma^\vee).
    $$
\end{ex}

\begin{rem}\label{rem:dependence-on-the-presentation}
    The object $\RHom_{\LCoMod_\Gamma}(\Gamma,A)$, or $\RHom_{\Gamma^\vee}(A,\Gamma^\vee)$, by definition, depends on the Hopf algebroid $(A,\Gamma)$, or the presentation $\Spec(A)\rightarrow \mathfrak{X}$ where $\mathfrak{X}$ is the stack associated to $(A,\Gamma)$. However, the object $\RHom_{\LCoMod_\Gamma}(\Gamma,A)$ or $\RHom_{\Gamma^\vee}(A,\Gamma^\vee)$ might sometimes be identified as canonical (at least for $\Spec(A)$): for example, consider the Hopf algebroid $(\C[x],\C[x,y]^\wedge_{(x-y)})=(\C[x],\C[x][[dx]])$; note that in this case we can not immediately apply Theorem \ref{thm:covariant-translation-for-profinitely-generated} since for $(\C[x],\C[x][[dx]])$ we should consider co-continuous comodules; yet following similar spirit we have the dual algebra $\C[x][\frac{d}{dx}]$, with the relations that $\frac{d}{dx}\circ f=f\circ \frac{d}{dx}+\frac{d}{dx}(f)$, and in this case we have 
    $$
    \RHom_{\C[x][\frac{d}{dx}]}(\C[x],\C[x][\frac{d}{dx}])\simeq \C[x][\frac{d}{dx}]/(\frac{d}{dx}\circ \C[x][\frac{d}{dx}])[-1],
    $$
    and there is an isomorphism of right $\C[x][\frac{d}{dx}]$-modules 
    $$
    C[x][\frac{d}{dx}]/(\frac{d}{dx}\circ \C[x][\frac{d}{dx}])\simeq \Omega^1_{\C[x]/\C},
    $$
    where $f\frac{d}{dx}$ acts on $\Omega^{1}_{\C[x]/\C}$ through $-\mathcal{L}_{f\frac{d}{dx}}$.

    Still, it is expected that, the object $\RHom_{\Gamma^\vee}(A,\Gamma^\vee)\otimes^L_{\Gamma^\vee}A\in D(\Z)$ depends only on the stack $\mathfrak{X}$ associated to $(A,\Gamma)$.
\end{rem}

\end{sloppy}
\end{document}